\documentclass[reqno,final]{amsart}
\usepackage{natbib}  
\usepackage{fancyhdr} 
\usepackage{color} 
\usepackage{hyperref} 
\usepackage{graphicx} 

\usepackage{pstricks}
\usepackage{amssymb}

\usepackage{hyperref}

\usepackage[math]{easyeqn}
\usepackage[mathscr]{eucal}

\usepackage{verbatim}

\definecolor{aleacolor}{rgb}{0.16,0.59,0.78}

\hypersetup{
breaklinks,
colorlinks=true,
linkcolor=aleacolor,
urlcolor=aleacolor,
citecolor=aleacolor}

\renewcommand{\headheight}{11pt}
\renewcommand{\cite}{\citet}

\theoremstyle{plain}
\newtheorem{theorem}{Theorem}[section]                                          
                          
\newtheorem{lemma}[theorem]{\bf{}Lemma}
\newtheorem{corollary}[theorem]{Corollary}

\theoremstyle{definition}
\newtheorem{definition}[theorem]{\bf{}Definition}
\theoremstyle{remark}
\newtheorem{remark}[theorem]{\bf{}Remark}

\makeatletter \@addtoreset{equation}{section} \makeatother
\renewcommand\theequation{\thesection.\arabic{equation}}
\renewcommand\thefigure{\thesection.\arabic{figure}}
\renewcommand\thetable{\thesection.\arabic{table}}

\def\muhat{\widehat{\mu}}
\def\hiprobt{1-e^{-t}}

\def\Fbb{D^2}
\def\Ftt{D^2_{u }}
\def\Fee{D^2_{v }}

\def\Fte{ D^2_{u v }}
\def\Fet{D^2_{v u}}

\def\rr{\mathtt{r}}

\def\zz{\mathfrak{z}}

\def\tt{\mathfrak{t}}

\def\rombt{\diamondsuit(\rr,t)}

\def\rb0{r_0^{\flat}}

\def\ND{\mathcal{N}}

\renewcommand{\Gamma}{\varGamma}
\renewcommand{\Pi}{\varPi}
\renewcommand{\Sigma}{\varSigma}
\renewcommand{\Delta}{\varDelta}
\renewcommand{\Lambda}{\varLambda}
\renewcommand{\Psi}{\varPsi}
\renewcommand{\Phi}{\varPhi}
\renewcommand{\Theta}{\varTheta}
\renewcommand{\Omega}{\varOmega}
\renewcommand{\Xi}{\varXi}
\renewcommand{\Upsilon}{\varUpsilon}

\def\Var{\operatorname{Var}}

\def\argmax{\operatornamewithlimits{argmax}}
\def\argmin{\operatornamewithlimits{argmin}}

\def\tr{\operatorname{tr}}

\def\R{I\!\!R}
\def\E{I\!\!E}
\def\P{I\!\!P}

\def\kappa{\varkappa}

\def\dLb12{T_h^{\flat}(\theta_1^{\flat}, \theta_2^{\flat})}

\def\localr{\Theta(\rr)}

\def\opttheta{\widehat{\theta}}

\def\alphab{\alpha^{\flat}}

\def\alphab12{\alpha^{\flat}(\theta, \theta_0)}
\def\chib12{\chi^{\flat}(\theta, \theta_0)}

\def\Lb0{L^{\flat}(\theta_0)}

\def\L0{L(\theta_0)}

\newcommand{\normp}[1]{
\left\Vert #1 \right\Vert
}

\newcommand{\vertiii}[1]{{\left\vert\kern-0.25ex\left\vert\kern-0.25ex\left\vert #1 
    \right\vert\kern-0.25ex\right\vert\kern-0.25ex\right\vert}}

\def\Ind{\operatorname{1}\hspace{-4.3pt}\operatorname{I}}

\def\xx{\mathtt{x}}

\def\etas{\eta^{*}}

\def\thetas{\theta^{*}}

\def\rr{\mathtt{r}}

\def\zz{\mathfrak{z}}

\def\tt{\mathfrak{t}}

\def\scorer{\breve{\nabla}}

\def\AA{A}

\def\etas{\eta^{*}}

\begin{document}

\title[Gaussian approximation for  penalized Wasserstein barycenters]{Gaussian approximation for  penalized Wasserstein barycenters
}

\author{Nazar Buzun}

\address{Affiliation \newline
Skolkovo Institute of Science and Technology,
Moscow.}

\email{n.buzun@skoltech.ru}


\begin{abstract}
 In this work we consider regularized Wasserstein barycenters (average in Wasserstein distance) in Fourier basis. We prove that random Fourier parameters of the barycenter converge to some Gaussian random vector by distribution. The convergence rate has been derived in finite-sample case with explicit dependence on measures count ($n$) and the dimension of parameters ($p$). 
\end{abstract}

\maketitle

\section{Introduction}

Monge-Kantorovich distance or Wasserstein distance is a distance between measures. It represents a transportation cost of measure $\mu_1$ into the other measure $\mu_2$.  
\begin{equation}
\label{was_orig_def}
W_p(\mu_1, \mu_2) =  \min_{\pi \in \Pi[\mu_1, \mu_2] } \left(  \int \| x - y \|^p d \pi (x, y) \right)^{1/p} 
\end{equation}
where the condition $\pi \in \Pi[\mu_1, \mu_2]$ means that $\pi(x,y)$ has two marginal distributions:  $\int_y d \pi(x,y) = d\mu_1(x)$ and $\int_x d \pi(x,y) = d \mu_2(y)$. We focus on regularized $W_1$ distance with probabilistic space $\{\R^d, \mathcal{B}(\| \cdot \|_2), L^1\}$ 
\[
\widetilde{W}_1(\mu_1, \mu_2) =  \min_{\pi \in \Pi[\mu_1, \mu_2] }  \int \| x - y \| d \pi (x, y)  + R_{\varepsilon}(\pi) 
\]where $R_{\varepsilon}(\pi)$ is a relatively small addition which improves differential properties of the distance. Namely without $R_{\varepsilon}(\pi)$ we can only bound the first derivative, with it we can bound the second derivative as well.    
There is a notion of mean in Wasserstein distance, called barycenter. And it is the main object in this research. Consider a set of random measures $\{ \mu_i \}_{i=1}^n$. By definition in regularised case the \textit{empirical  and reference barycenters} are  
\[
\muhat = \argmin_{\mu} \sum_{i=1}^n \widetilde{W}_1(\mu, \mu_i)
\]
and
\[
\mu^* = \argmin_{\mu} \sum_{i=1}^n \E \widetilde{W}_1(\mu, \mu_i)
\]
Barycenters are center-of-mass generalization. If we look at the barycenter of a set of uniform measures it extracts the common ``shape'' form of these measures. If the measures are sampled from some distribution then their barycenter can be treated as an empirical approximation of the distribution mean. A simple example is a circles set with means $\{m_i \in \R^2\}$ and radius's $\{r_i\}$.
\begin{figure}
\begin{center}
\includegraphics[scale=0.6]{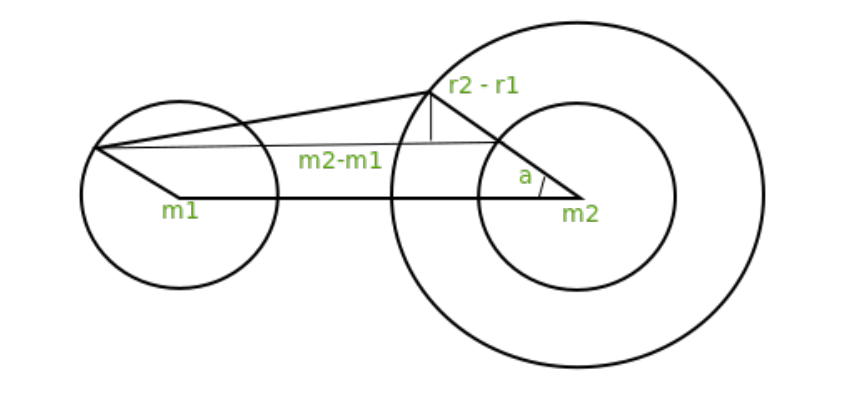}
\end{center}
\caption{Illustration for $W_2$ distance computation between two circles $(m_1, r_1)$ and $(m_2, r_2)$.}
\end{figure}
\[
W_2^2 \big((m_1, r_1), (m_2, r_2) \big) 
\]
\[
= \frac{1}{2\pi} \int_{0}^{2\pi}  [(m_2 - m_1) - (r_2 - r_1) \cos(a)]^2 + [(r_2 - r_1) \sin(a)]^2 da  
\]
\[
= (m_2 - m_1)^2 + (r_2 - r_1) ^2
\]
Their $W_2^2$ barycenter is also a circle with mean $m = \frac{1}{n} \sum_{i=1}^n m_i$ and radius $r = \frac{1}{n} \sum_{i=1}^n r_i$. We  refer to papers \cite{BinWas}, \cite{PenWas}, \cite{bc_coords_proj} for an overview of the barycenters and related study. 

$W_1$ barycenter in the previous example doesn't have an explicit formula but we will show below that it has Gaussian approximation.     
It is well known that the center-of-mass in $l_2$ norm converges to a Gaussian random vector. As for the barycenter ($\muhat$), it is also expected to have some Gaussian properties. For example, if the measures are Gaussian themselves or one-dimensional or circles set then the Gaussian approximation of the $W_2^2$ barycenter is proven in papers \cite{BinWas}, \cite{g_was_clt}. In circles set case the mean and radius converge to some Gaussian variables as a sum of independent observations according to Central Limit Theorem.  In one-dimensional case, denoting distribution functions by $F_i(x)$
\[
W_2^2(\mu_1, \mu_2)  =  \int_0^1 | F_1^{-1}(s) - F_2^{-1}(s) |^2 ds
\]  
one gets
\[
\widehat{F}^{-1}(s) = \frac{1}{n} \sum_{i=1}^n F_i^{-1}(s)
\]
In the case of Gaussian measures with zero mean and variances $\{S_i\}$ 
\[
 W_2^2(\mu_1, \mu_2) =  \tr \{S_1\} + \tr \{S_2 \} - 2 \tr \{ (S_2^{1/2}S_1S_2^{1/2})^{1/2} \}
\] 
and for some non-random matrix $S_*$ (ref. \cite{LimitWasG}) the corresponded barycenter variance is 
\[
\widehat{S} = \frac{1}{n} \sum_{i=1}^n (S_*^{1/2}S_iS_*^{1/2})^{1/2} + O(1/n)
\]
In both last examples one deals with a mean of independent random variables. Being multiplied by $\sqrt{n}$ factor, they converge to a Gaussian variable (or to a Gaussian process in case of $\widehat{F}^{-1}(s)$ by Donsker's Theorem). 
In general case it appears to be very difficult  to reveal such convergence because the barycenter doesn't have an explicit equation and it is an infinite-dimensional object. In order  to handle with this difficulty  we propose an approximation of the $W_1$ barycenter by a sum of independent variables using projection into Fourier basis and involve some novel results from statistical learning theory. The perspective of Fourier Analysis provides a suitable representation of the Wasserstein distance and it is already studied in the literature \cite{FouWas}. Denote a range of size $p$ of the barycenter Fourier coefficients by
\[
\opttheta_p =  \mathcal{F}_p \left( \frac{d\muhat(x)}{dx} \right),
\quad
\thetas_p =  \mathcal{F}_p \left( \frac{d\mu^*(x)}{dx} \right)
\]  
The first our result states that for some non-random matrix $\breve{D}$, independent random vectors $\{ \xi_i \}$, and some depending on $p$ constant       
\[
\left\| \breve{D} \left( \opttheta_p  - \thetas_p \right) - \sum_{i=1}^n \xi_i \right\|  = \frac{C(p)}{\sqrt{n}}
\]
Further we show that for some Gaussian vector $Z$  
\[
W_1 \left( \breve{D} ( \opttheta_p  - \thetas_p ), \, Z \right ) = \frac{C(p)}{\sqrt{n}}
\] 
and $\forall z \in \R_+$:
\[
\left | P \left( \|  \breve{D} ( \opttheta_p  - \thetas_p ) \| > z \right) -  P \left(\|Z\| > z \right) \right | = \frac{C(p)}{\sqrt{n}}
\]    
\textit{Statistical Application:} The last statement allows us to obtain the confidence region of parameter $\opttheta_p$ and describe the distribution inside the region. Besides, the bootstrap procedure  validity \cite{LimitWas} follows from our proof as well. If one samples $\|  \breve{D} ( \theta_p^{boot} -  \opttheta_p ) \|$ using bootstrap it would be close by quantiles to the random variable $\| \breve{D} ( \opttheta_p - \thetas_p ) \|$,  which also relates to the formation of the confidence region. In  \cite{bc_coords_proj} the authors demonstrate application of barycentric coordinates that allow to infer missing geometry of an input mesh using a set of 3D models. Recent article \cite{brule} shows empirically that barycenters may be helpful as a loss function in unsupervised face landmarks detection task.      

The Structure of this paper is the following. The main Theorems are in Section 2. Section 3 deals with independent parametric models and describes how one can approximate parameter deviations by a sum of independent random vectors $\{ \xi_i \}$. 
In Section 4 we explore the barycenters model, compute derivatives of the Wasserstein distance using infimal convolution of support functions and check the required assumptions from the 3-rd Section.
Section 5 contains some useful properties of the support functions.
The final part, Gaussian approximation of the parameter $\opttheta_p$, is completed in Section 6, where we prove that $\{ \xi_i \}$ is close to $Z$ by distribution.

\section{The main result} \label{main_res}

\definition[W-dual]{For two random variables $X$ and $Y\in\mathbb{R}^{\infty}$ with
densities $\varphi_{X}$ and $\varphi_{Y}$ define Wasserstein distance in dual form as
\[
W_1(\varphi_X, \varphi_Y)=\max_{\forall x: \|\nabla f(x)\|\leq 1}\{\E f(X)-\E f(Y)\}
\]
where $\forall x: \|\nabla f(x)\|\leq 1 $ means that function $f$ is $1$-Lipshits.
}
Note that if $\pi(x,y)$ is a joint distribution with marginals $\varphi_{X}$ and $\varphi_{Y}$ then this definition is equivalent to  the original (\ref{was_orig_def}), 
which follows from Kantorovich-Rubinstein duality \cite{KRdual}.  
 \definition[W-dual-regularised]{For two random variables $X$ and $Y\in\mathbb{R}^{\infty}$ with
densities $\varphi_{X}$ and $\varphi_{Y}$ and some density function $G(x)$  define a regularised Wasserstein distance in dual form as
\[
\widetilde{W}_1(\varphi_ X, \varphi_ Y)=\max_{\forall x: \|\nabla f(x)\|\leq 1} \left\{\E f(X)-\E f(Y) -   \varepsilon \int  \| \nabla f (x) \|^2  G(x) dx   \right\}
\]
} 
The regulariser term in this definition allows to bound the second derivative of the distance which will be shown below.

Consider a set of random measures (random measure is a measure-valued random element) with densities $\phi_1, \ldots, \phi_n$.  Let the barycenter measure $\muhat$ has density $\widehat{\phi}$ and Fourier coefficients $\opttheta = \theta(\widehat{\phi}) \in \R^{\infty}$. 
Denote Fourier coefficients of the other measures $\forall i: \theta_i = \theta(\varphi_i) \in \R^{\infty}$. 
Define an independent parametric model with dataset ($\theta_1, \ldots, \theta_n$) and parameter $\theta$.
\begin{equation}\label{bc_model}
L(\theta) = - \sum_{i=1}^n  \widetilde{W}_1\big(\phi(\theta), \, \phi_i(\theta_i) \big)
\end{equation}

\noindent Note that the model's likelihood  has minus sign and there is no normalising factor $1/n$.  According to notation of the previous Section the reference parameter value (coefficients of the  reference  barycenter) is 
\[
\theta^{*}=\argmax_{\theta}\E L(\theta)
\]
Define a local region around $\thetas$
\[
\Omega(\rr) = \{ \theta: \|D (\theta - \thetas) \| \leq \rr \}
\]
where $D$ is  Fisher matrix of the model
\[
D^2 = - \nabla^2 \E L(\thetas)
\]
Here and below operator $\nabla$ denotes Fréchet derivative. We also assume below that matrix $D$ is invertible. This assumption is not strict, since the same results can be obtained with the pseudo-inverse matrix $D^+$. Since we do not divide the model by factor $n$, the  matrix $D^2$ has an order $O(n)$.

\begin{theorem} \label{wasF} Assume that exists a generalised differentiable Fourier basis $\{\psi_k(x)\}_{k=1}^{\infty}$  with Gram matrix $G(x)$ in which for some constant $C_{D}$ and minimal eigenvalue it holds
\[
\lambda_{\min} \big( D K\circ G D \big) \geq C_{D} n
\]
where 
\[
K\circ G  = \int \nabla^T  \psi(x) \nabla \psi^T(x) G(x) dx 
\]
Let the random Fourier parameters of the dataset have a common density: $\theta_1 \ldots \theta_n \sim q(\theta)$, and it fulfills condition 
\[
  \int_{\R^{\infty}}   \| D^{-1} \nabla q(\theta) \| d\theta  = \frac{C_q}{\sqrt{n}} 
\]
Let $\opttheta, \thetas \in \R^{\infty}$ be the Fourier coefficients of the empirical and reference barycenter defined above and $\varepsilon$ be the regulariser constant of $\widetilde{W}_1$. Then with probability $\hiprobt$ 
\[
\normp{ D (\opttheta - \theta^*) -  D^{-1} \nabla L (\thetas) } 
\leq  \rombt
\]\[
\rombt \leq  \frac{ O(H_1 + t)   }{\varepsilon C^{3/2}_{D} \sqrt{n}} +   \frac{  O(t  C_q ) }{\varepsilon C^2_{D} \sqrt{n}}   + O \left(\frac{H_2}{n}\right)
\]
where $H_1$, $H_2$  are components of ellipsoid entropy (ref. Section \ref{entropy}) with matrix $D$ and eigenvalues~$\{\lambda_i\}_{i=1}^{\infty}$, such that for some absolute constant $C$
\[
H_1
\leq C \sqrt{\alpha - 1} \sqrt{\sum_i \frac{ \log^\alpha (\lambda_i^2 / \lambda^2_{\min}) }{\lambda_i^{2} / \lambda^2_{\min} }  }
\]
and
\[
H_2
\leq C  \sum_i \frac{ \lambda_{\min} }{\lambda_i}  
\]
\end{theorem}

\noindent We have shown that deviations of the parameter $\opttheta$ may be approximated by expression $D^{-1} \nabla L (\thetas)$ which is a sum of independent random vectors. Next one may derive Gaussian approximation for the last term and find the  correspondent normal vector $Z$.   
 Define additional Fisher matrix corresponded to the projection into the first $p$ elements of the parameter $\theta$ (we will describe it in more detail  in Section \ref{slt}).
\[
{\breve{D}}^2  =  D^2_{p\times p} - D^2_{p\times \infty} D^{-2}_{\infty\times \infty} D^2_{\infty\times p} 
\]
such that 
\[
D^2 =  \left(
\begin{matrix}
    D^2_{p \times p} & D^2_{p \times \infty} \\
    D^2_{\infty \times p} & D^2_{\infty \times \infty} 
\end{matrix} \right)
\]
and define the gradient of the projection into the first $p$ elements of the parameter~$\theta$. 
\[
\scorer = \nabla_{1 \ldots p} -   D^2_{p\times \infty} D^{-2}_{\infty\times \infty} \nabla_{p \ldots \infty}
\]
 The next result shows that vector $\opttheta$ is close by distribution to some Gaussian vector 
\[
 Z \sim \ND(0,  \Var [{\breve{D}}^{-1} \scorer L(\thetas)])
\]
\begin{theorem} \label{was_gar} Let $\opttheta_p, \thetas_p \in \R^p$ are  the first $p$ Fourier coefficients of the empirical and reference barycenters. Then under assumptions from  Theorem \ref{wasF}  with probability $(\hiprobt)$ 
\[
W_1 ({\breve{D}} (\opttheta_p - \thetas_p), \, Z) \leq \rombt +     O\left ( \frac{p \log n}{ \sqrt{C_{D} n} }   \right) 
\]and  $\forall z \in \R_{+}$
\[
\left| \P ( \| \breve{D} (\opttheta_p - \thetas_p) \| > z )  -  \P ( \| Z \| > z ) \right| 
 \leq  O\left ( \frac{\rombt}{\sqrt{p}} + \frac{\sqrt{p} \log^2 n}{ \sqrt{C_{D} n} }  \right) 
\]and for $ \rombt$ we provide an upper bound in Theorem \ref{wasF}.  
\end{theorem}

\section{Statistical learning theory }
\label{slt}
In this section we consider an infinite dimensional statistical model $L (\theta) $. Let parameter $\theta $ consists of two parts $(u, v) $, such that $u=\theta_{1\ldots p} \in \R^p$. Working with a finite dataset of size $n$ we are going to find the deviations of $\argmax L(\theta)$ basing on three model assumptions listed below. Further we will make some specifications for independent models and check that the barycenter model (\ref{bc_model}) satisfies to these assumptions. 
\subsection{General approach}
 
The Likelihood function $L(\theta) = L(\theta, \mathbb{Y})$  depends on the parameters vector $\theta = (u,v)  $ and a random dataset $\mathbb{Y}$ of size $n$. Denote parameter's MLE and reference values:
\[
\opttheta= \argmax_\theta L(\theta )
\]
\[
\thetas = \argmax_\theta \E L(\theta)
\]
We are going to study deviations of $\opttheta$ and $u$ in the following sense. For Fisher matrix $D^2$  and some projection matrix defined below  $\breve{D}^2$    the deviation $D(\opttheta - \thetas)$  may be approximated by $D^{-1} L(\thetas) $ and analogically $\breve{D} (\widehat{u} - u^*) $ by $ \breve{D}^{-1} \scorer L (\thetas)$. 
This approach bases on paper \cite{spok_pen}. Remind that by definition Fisher matrix~is 

\begin{equation}\label{matF}
D^2 =- \nabla^2 \E L(\thetas) =  \left(
\begin{matrix}
    \Ftt & \Fte \\
    \Fet & \Fee 
\end{matrix} \right)
\end{equation}
Here and below operator $\nabla$ means Fréchet derivative. We also assume below that matrix $D$ is invertible. Denote the stochastic part of the Likelihood 
\[
\zeta(\theta) = L(\theta) - \E L(\theta)
\]
It would be easier to deal with the model  if the matrix $\Fbb$ has block-diagonal view ($\Fte = \Fet= 0$). One can make parameter replacement in order to satisfy to this condition.  Define a new parameter $\vartheta = \vartheta(u,  v)$ such that
\[
\nabla_u \nabla_\vartheta^T \E L(\thetas) = \nabla_\vartheta  \nabla_u^T \E L(\thetas) = 0
\]
and 
\[
\vartheta = v +  D^{-2}_{v} \Fet u
\]
Or in other words the parameter's transformation matrix is	  
\[
 S = 
\left(\begin{matrix}
    I & 0 \\
     D^{-2}_{v} \Fet & I 
\end{matrix} \right),
\quad
 S^{-1} = 
\left(\begin{matrix}
    I &  0  \\
    - D^{-2}_{v}\Fet  & I 
\end{matrix} \right)
\]
The gradient in the new coordinates $(u, \vartheta)$  may be obtained by rule $\nabla(u, \vartheta)= (S^{-1})^T \nabla(u,  v)$. Use notation $\scorer$ for its first part  
\[
\scorer = \nabla_{u} (u, \vartheta) = \nabla_{u} -   \Fte  D_{v}^{-2} \nabla_{v}
\]
Fisher matrix after parameters replacement changes by rule \\ $\Fbb(u, \vartheta) =  (S^{-1})^T \Fbb S^{-1}$,  so in the new coordinates it has view
\[
D^2(u, \vartheta) = - \nabla^2 \E L(u^*, \vartheta^*) =  \left(
\begin{matrix}
    \breve{D}^2 & 0 \\
    0 & D_{\vartheta}^2 
\end{matrix} \right)
\]
\[
\breve{D}^2  = \Ftt - \Fte D_v^{-2} \Fet 
\] The last equation clarifies the ``strange'' structure of the projection matrix $\breve{D}^2 $. It makes parameter $u$ independent on the other part of $\theta$.   Define a local region around point $\thetas$ 
\begin{equation}\label{localR}
\Omega(\rr) = \{\theta: \| D (\theta - \thetas) \| \leq \rr \}
\end{equation}   
Now we will write down three conditions on the Likelihood derivatives which are essential for the deviations of $\opttheta$. The first and second conditions should be satisfied in the local region $\Omega(\rr)$. The third condition is required to make expansion of the previous two conditions to the whole parameter space $\R^{\infty}$. Further we will show that these conditions are also sufficient for deviation bounds of the parameter $u$, namely from  deviations bound of $\opttheta$ follows analogical bound for $u$.

\noindent\textbf{Assumption 1:} In the region $\Omega(\rr)$
\begin{EQA}
\normp{ - D^{-1}  \{ \nabla \E L(\theta) -  \nabla \E L(\thetas) \} - D (\theta - \thetas) } & \leq & \delta(\rr) \rr
\end{EQA}
\textbf{Assumption 2:} In the region $\Omega(\rr)$ with probability $\hiprobt$
\begin{EQA}
\sup_{\theta \in \Omega(\rr)} \normp{  D^{-1}  \{ \nabla \zeta (\theta) - \nabla \zeta (\thetas) \}  } & \leq &  \zz(t) \rr
\end{EQA}
\textbf{Assumption 3:}
The Likelihood function is convex  $(-\nabla^2 L(\theta) \geq 0)$ or the expectation of Likelihood function is upper-bounded by a strongly convex function $(\exists b > 0: \,  \E L(\thetas) - \E L(\theta) \geq b \| D (\theta - \thetas) \|^2$).

\noindent Use notation 
\begin{equation}\label{romb}
  \rombt =   \{\delta(\rr)  +  \zz(t)\} \rr
\end{equation}
\begin{theorem}
\label{devbound}
Let the Likelihood function be convex $(-\nabla^2 L(\theta) \geq 0)$ and for $\rr \leq \rr_0$  it holds $\delta(\rr)  +  \zz(t) \leq 1/2$. Then under Assumptions 1,2 with probability $\hiprobt$

\[
\| D (\opttheta - \theta^*) -  D^{-1} \nabla L (\thetas) \| 
\leq  \rombt
\]
\[
\| \breve{D} (\widehat{u} - u^*) -  \breve{D}^{-1} \scorer L (\thetas) \| 
\leq \rombt
\]
and 
\[
\rr \leq \rr_0 = 4 \|D^{-1} \nabla L(\thetas) \|
\]
\end{theorem}

\begin{proof} From $(-\nabla^2 L(\theta) \geq 0)$ and ($L(\opttheta) > L(\thetas)$) follows that the local region $\Omega(\rr)$ that includes $\opttheta$ should cover the next region 

\[
\Omega(\rr)  \supset \{\theta: L(\theta) \geq L(\thetas) \}
\]Use notation 
\[
D^2(\theta) = - \nabla^2 \E L(\theta)
\]
Estimate the minimal possible radius $\rr_0$ that satisfy to the previous condition. Let $\theta_0$ be some point between $\theta$ and $\thetas$ that is used in Taylor expansion with central point $\thetas$.
\[
0 \geq L(\thetas) - L(\theta) 
\]\[
= -(\theta-\thetas)^T \nabla L(\thetas) - \frac{1}{2} (\theta - \thetas)^T \nabla^2 \zeta(\theta_0) (\theta - \thetas)  + \frac{1}{2} \| D(\theta_0) (\theta - \thetas) \|^2
\]Assumption 1 provides 
\begin{align*}
 \| D(\theta_0) (\theta - \thetas) \|^2 
& =\| D (\theta - \thetas) \|^2 + \| (D(\theta_0) - D) (\theta - \thetas) \|^2 \\
& \geq  \rr^2 - \delta(\rr) \rr^2
 \end{align*}
Assumption 2 provides with probability  ($\hiprobt$)
\[
(\theta - \thetas)^T \nabla^2 \zeta(\theta_0) (\theta - \thetas) \leq \zz(t) \rr^2
\]
Put these two properties into the initial inequality 
\[
0 \geq - \|D^{-1} \nabla L(\thetas) \| \rr - \frac{ \zz(t) }{2} \rr^2 + \frac{1-\delta(\rr)}{2} \rr^2   
\]
\[
\rr (1 - \delta(\rr) - \zz(t)) \leq 2 \|D^{-1} \nabla L(\thetas) \|
\]So one may set under the assumption $\delta(\rr)  +  \zz(t) \leq 1/2$
\[
\rr_0 = 4 \|D^{-1} \nabla L(\thetas) \|
\]From Assumptions 1,2 also  follows that
\[
\| D (\opttheta - \theta^*) +  D^{-1} \{\nabla L (\opttheta) -  \nabla L (\thetas) \}   \| 
\leq  \rombt
\]
Since $\nabla L (\opttheta) = 0$ we have
\[
\| D (\opttheta - \theta^*) -  D^{-1} \nabla L (\thetas) \| 
\leq  \rombt
\]Note that for the coordinates transform $S$  there exists  the following \textit{invariant}:
\[
\left\| 
\left(
\begin{matrix}
    \breve{D} & 0 \\
    0 & D_{\vartheta} 
\end{matrix} \right)
\left( \begin{matrix}
    u  - u^* \\
    \vartheta - \vartheta^* 
\end{matrix} \right)
 +
 \left(
\begin{matrix}
    \breve{D}^{-1} & 0 \\
    0 & D_{\vartheta}^{-1} 
\end{matrix} \right) 
\left( \begin{matrix}
    \scorer L(u,  \vartheta)  - \scorer L(u^*,  \vartheta^*)\\
    \nabla_\vartheta L(u,  \vartheta)  + \nabla_\vartheta L(u^*,  \vartheta^*) 
\end{matrix} \right)
\right\|
\]
\[
  = \|  D(\theta - \thetas) +  D^{-1} \{\nabla L (\theta) - \nabla L (\thetas) \} \|  
\]
Since 
\[
\left\| 
\left(
\begin{matrix}
    \breve{D} & 0 \\
    0 & D_{\vartheta} 
\end{matrix} \right)
\left( \begin{matrix}
    u  \\
    \vartheta  
\end{matrix} \right) 
\right \| ^2  =  \theta ^T S^T [(S^{-1})^T D^2 (S^{-1})] S \theta  = \| D  \theta  \|^2 
\]
\[
\left \| \left(
\begin{matrix}
    \breve{D}^{-1} & 0 \\
    0 & D_{\vartheta}^{-1} 
\end{matrix} \right) 
\left( \begin{matrix}
    \scorer \\
    \nabla_\vartheta  
\end{matrix} \right)
\right\|^2 =
\nabla^T S^{-1}   [(S^{-1})^T D^2 (S^{-1})] ^{-1}  (S^{-1})^T \nabla
\]
\[
=  \nabla^T D^{-2} \nabla
\]
\[
\left( \begin{matrix}
    \scorer \\
    \nabla_\vartheta  
\end{matrix} \right)^T \left( \begin{matrix}
    u  \\
    \vartheta  
\end{matrix} \right)  = \nabla^T S^{-1} S \theta = \nabla^T \theta
\]Subsequently, basing on this invariant we obtain the bound for projection $u$
\[
\| \breve{D} (\widehat{u} - u^*) -  \breve{D}^{-1} \scorer L (\thetas) \| 
\]
\[
  \leq \| D (\opttheta - \theta^*) -  D^{-1} \nabla L (\thetas) \| \leq \rombt
\]

\end{proof}

\begin{remark} The case with a non-convex function $L$ is considered in \cite{spok_pen}, where the author involve an upper-bound convex function for $\E L$ decline (Theorem 2.1).  
\end{remark}

\subsection{Ellipsoid entropy}
\label{entropy}

\noindent The upper bound $\zz(t)$ of the random process in \textbf{Assumption 2} with parameter $\theta \in \Omega(\rr)$ requires entropy computation of the ellipsoid $\Omega(\rr)$. It will be useful to us  in the next section and below we provide a short excerpt on this topic. The general formula for the covering number $N(\varepsilon, \Omega)$ of a convex set $\Omega$ in $\mathbb{R}^{p}$ with Euclidean metric is
\[
N(\varepsilon,\Omega)\leq\frac{volume(\Omega+(\varepsilon/2)B_{1})}{volume(B_{1})}\left(\frac{2}{\varepsilon}\right)^{p}
\]where $B_1$ is the unit ball. Remind that $N(\varepsilon,\Omega)$ equals to the minimal count of balls with radius $\varepsilon$ that is sufficient to cover $\Omega$.   We will need two components of the ellipsoid entropy:
\begin{equation}\label{Hdef}
H_1 =  \int_{0}^{1 / 2}\sqrt{ \log N(\varepsilon \rr, \Omega(\rr))}d\varepsilon,
\quad
H_2 =  \int_{0}^{1 / 2} \log N(\varepsilon \rr, \Omega(\rr)) d\varepsilon
\end{equation}

\begin{lemma}\label{Hbound} Let $\| D^{-2} \| = 1$. Then for the entropy components (\ref{Hdef}) of the ellipsoid $\Omega(\rr)$ with matrix $D^2$ defined in expression (\ref{localR}) it holds

\[
H_1
\leq C \sqrt{\alpha - 1} \sqrt{\sum_i \frac{ \log^\alpha (\lambda_i^2(D)) }{\lambda_i^{2}(D)}  }
\]
and
\[
H_2
\leq C  \sum_i \frac{ 1 }{\lambda_i(D)}  
\]where $C$ is some absolute constant. 

\end{lemma}  

\begin{proof}
Function $\log N(\varepsilon \rr, \Omega(\rr))$ is monotone-decreasing. One may split the integration interval of  (\ref{Hdef}) into the following parts:
\[
\left[\frac{1}{2}, \frac{1}{4}\right], \quad \left[\frac{1}{4}, \frac{1}{8}\right], 
\quad \left[\frac{1}{8}, \frac{1}{16}\right] , \quad \ldots 
\]Take corresponded  values $N(\rr/4, \Omega(\rr))$, $N(\rr/8, \Omega(\rr))$, $N(\rr/16, \Omega(\rr))$, $\ldots$ and obtain integral approximation by the histogram 
\[
H_1 \leq  \sum_{k=2}^{\infty} \frac{1}{2^k} \sqrt{\log N \left(\frac{\rr}{2^k}, \, \Omega(\rr) \right)}
\]
and 
\[
H_2 \leq  \sum_{k=2}^{\infty} \frac{1}{2^k} \log N \left(\frac{\rr}{2^k}, \, \Omega(\rr) \right)
\]
Theorem H.7.1 in \cite{spok_pen} provides upper bounds for the right parts of the previous expressions and completes the proof. 
\end{proof}

\subsection{Independent models}

Consider independent models (models with independent observations) and obtain a simpler variant of  \textbf{Assumption 2}. Involve three basic Lemmas for that.

\begin{lemma}[Bernstein\textquoteright s inequality Theorem 2.10  \cite{Concentration}]\label{bern}
Let $X_{1}\ldots X_{n}$ be independent real-valued random variables.
Assume that exist positive numbers $\mathbf{v}$ and $R$ such that 
\[
\mathbf{v^{2}=}\sum_{i=1}^{n}\E X_{i}^{2}
\]
and for all integers $q\geq3$
\[
\sum_{i=1}^{n}\E\left[X_{i}\right]_{+}^{q}\leq\frac{q!}{2}\mathbf{v^{2}}R^{q-2}
\]
Then for all $\lambda\in(1,1/R)$ 
\[
\log\E e^{\lambda\sum_{i}(X_{i}- \E X_{i})}\leq\frac{\mathbf{v}^{2}\lambda^{2}}{2(1-R\lambda)}
\]

\end{lemma}

\begin{lemma}[Dudley\textquoteright s entropy integral Lemma 13.1 \cite{Concentration}] 
\label{dudl}
Let $\Omega$ be a finite pseudometric space and let $f(\theta)$ ($\theta\in \Omega$) be a collection of random variables such that for some constants $a,\mathbf{v},R>0$,
for all $\theta_1,\theta_2 \in \Omega$ and all $0 < \lambda < (Rd(\theta_1,\theta_2))^{-1}$
\[
\log\E \exp \{\lambda(f(\theta_1) -  f(\theta_2)) \}\le a\lambda d(\theta_{1}, \theta_{2})+\frac{\mathbf{v^{2}}\lambda^{2}d^{2}(\theta_{1}, \theta_{2})}{2(1 - R\lambda d(\theta_{1}, \theta_{2}))}
\]
Then for any $\theta_{0}\in \Omega$, 
\[
\E[\sup_{\theta} f(\theta) - f(\theta_{0}) ]\leq3a\rr+12\mathbf{v}\int_{0}^{\rr/2}\sqrt{\log N(\varepsilon,\Omega)}d\varepsilon+12R\int_{0}^{\rr/2}\log N(\varepsilon,\Omega)d\varepsilon
\]
where $\rr=\sup_{\theta \in \Omega}d(\theta, \theta_{0})$ and $N(\varepsilon, \Omega)$ is covering number.

\end{lemma}

\begin{lemma}[Bousquet's inequality Theorem 12.5 \cite{Concentration}] \label{buske} Consider independent random
variables $X_{1}\ldots X_{n}$ and let  $\mathcal{F}:X\rightarrow\mathbb{R}$
be countable set of functions that satisfy conditions $\E f(X_{i})=0$
and $\|f\|_{\infty}\leq R$. Define 
\[
Z=\sup_{f\in\mathcal{F}} \sum_{i=1}^{n}f(X_{i})
\]
Let $\textbf{v}^{2}\geq  \sup_{f \in \mathcal{F}}  \sum_{i=1}^{n}\E f^{2}(X_{i})$
then with probability $\hiprobt$
\[
Z<\E Z + \sqrt{2t(\textbf{v}^{2}+2R\E Z)}   +\frac{t R}{3}
\]
\end{lemma}
\noindent Apply these lemmas in order to simplify  \textbf{Assumption 2} for  independent models. The likelihood  of  independent model is a sum of independent functions:  
\[
(L-\E L)(\theta) = \zeta(\theta) =  \sum_{i=1}^n \zeta_i(\theta) 
\]
Note that $\zeta_i $ depends from the implicit i-th element from the dataset, such that $\zeta_i(\theta) = \zeta_i(\theta, Y_i)$.    
\begin{theorem} \label{ass2cond} Let $D$ be the  matrix from Assumption 2 and  $\forall \theta \in \Omega(\rr)$ (ref. definition~\ref{localR})
\[
\sup_{\| u \| = 1} \sum_{i=1}^n \E (u^T D^{-1} \nabla^2 \zeta_i (\theta) D^{-1} u)^2 \leq \textbf{v}^2 
\]
and
\[
\| D^{-1} \nabla^2 \zeta_i (\theta) D^{-1} \| \leq R
\]
Then Assumption 2 fulfills inside $\Omega(\rr)$ with probability $\hiprobt$ and
\[
\zz(t) \leq  \textbf{v} (12 \sqrt{2} H_1 + \sqrt{2t}) + R (24 H_2 + 24 \sqrt{H_2 t} + t/3)
\]where $H_1$ and $H_2$ are the components of ellipsoid entropy (\ref{Hdef}).
\end{theorem}
\begin{proof}
Set a random process for each $i$: 
\[
X_i(\gamma, \theta)  = \frac{1}{\rr} \gamma^T \{ \nabla\zeta_i(\theta) -  \nabla\zeta_i(\thetas) \}
\]
Such that
\[
 \sup_{\|D \gamma\| \leq \rr} \sum_i X_i(\gamma, \theta) = \| D^{-1} \{ \nabla\zeta(\theta) -  \nabla\zeta(\thetas) \} \|
\]
$\forall \, \text{fixed} \, (\gamma, \theta) \in \Omega(\rr, 0) \times \Omega(\rr, \thetas)$ and $\| u \| = 1$:
\[
\sup_u \E \sum_i (\nabla_{\theta} X_i(\gamma, \theta)^T D^{-1} u)^2 
= \sup_u  \E \sum_i \frac{1}{\rr} ( \gamma \nabla^2 \zeta (\theta)^T D^{-1} u)^2 
\]
\[
\leq \sup_u \E \sum_i ( u^T D^{-1} \nabla^2 \zeta (\theta)^T D^{-1} u)^2
\leq \textbf{v}^2 
\]Analogically
\[
\sup_u \E \sum_i (\nabla_{\gamma} X_i(\gamma, \theta)^T D^{-1} u)^2 \leq \textbf{v}^2 
\]
$\forall i \in {1, \ldots ,n} :$
\[
 \| D^{-1} \nabla X_i(\gamma, \theta) \| \leq R 
\]
Apply Lemma \ref{bern} for the sum of random variables $X (\gamma, \theta) =  \sum_i X_i(\gamma, \theta)$ when $(\gamma, \theta)$ are fixed.
\begin{EQA}
&& \log  \E \exp \lambda \left( X(\gamma_1, \theta_1) - X(\gamma_2, \theta_2) \right) \\
&& = \log  \E \exp \lambda \left( (\gamma_1 - \gamma_2)^T \nabla_{\gamma} X(\gamma, \theta) \right) + \log  \E \exp \lambda \left( (\theta_1 - \theta_2)^T  \nabla_{\theta} X(\gamma, \theta) \right)  \\
&& \leq \sup_u \log  \E \exp \lambda \left( \| D(\gamma_1 - \gamma_2) \| u^T D^{-1} \nabla_{\gamma} X(\gamma, \theta) \right) \\ 
&& + \sup_u \log  \E \exp \lambda \left( \| D (\theta_1 - \theta_2) \| u^T D^{-1}  \nabla_{\theta} X(\gamma, \theta) \right)  \\
&& \leq     \frac{\mathbf{v}^{2}\lambda^{2} \| D (\gamma_2 - \gamma_1) \|^2  }{2(1-R\lambda  \| D (\gamma_2 - \gamma_1) \|  )} + \frac{\mathbf{v}^{2}\lambda^{2} \| D(\theta_2 - \theta_1) \|^2 }{2(1-R\lambda \| D(\theta_2 - \theta_1) \|)} \\
&& \leq  \frac{\mathbf{v}^{2}\lambda^{2} d^2_{12} }{2(1-R\lambda d_{12})}
\end{EQA}
\[
d^2_{12} = \| D(\theta_2 - \theta_1) \|^2  + \| D( \gamma_2 - \gamma_1) \|^2 
\]Denote 
\[
\Upsilon = \Omega(\rr) \times \Omega(\rr)
\]
such that $\log N(\varepsilon, \Upsilon) = 2 \log N(\varepsilon, \Omega (\rr))$. Then with Lemma \ref{dudl} we obtain 
\[
\E \sup_{\gamma, \theta} X(\gamma, \theta) \leq  12 \mathbf{v}  \int_{0}^{\rr/ 2}\sqrt{\log N(\varepsilon, \Upsilon)}d\varepsilon
+ 12R  \int_{0}^{\rr/ 2}\log N(\varepsilon, \Upsilon)d\varepsilon
\] 
Application of Lemma \ref{buske} to the random variable $Z = \sup_{\gamma, \theta} X(\gamma, \theta)$ leads to the final result. Making variable replacement $\varepsilon = \rr \varepsilon'$ and dividing the last expression by $\rr$ we obtain
\[
\zz(t) \leq  E +  \sqrt{2t(\textbf{v}^{2}+2R E)}   +\frac{tR}{3}  
\]
where 
\[
E =  12 \sqrt{2} \mathbf{v}  H_1
+ 24R   H_2
\]
In order to simplify this inequality we have also used property $H_1^2 \leq H_2$. 
\end{proof}

\section{Barycenters model} \label{bc}

  We are going to show that \textbf{Assumptions 1,2,3} are fulfilled for the barycenters model (\ref{bc_model}) defined in Section \ref{main_res}. Also we need to estimate $\rombt$ from expression (\ref{romb}). It will allow us to apply Theorem \ref{devbound}.  Remind that we deal with the likelihood function $L(\theta) = L(\theta, \{\theta_i\}_{i=1}^n)$, where the  implicit random vectors $\{\theta_i\}_{i=1}^n$ is the dataset of Fourier coefficients corresponded to the  random measures~$\{ \mu_i\}_{i=1}^n$.

We start with derivatives study of the model and  show that Wasserstein distance in some differentiable Fourier basis $\{ \psi_i(x) \}_{i=1}^{\infty}$ is a support function (ref. Def(s) from Section \ref{supf}). Consider Wasserstein distance in the dual form (ref. definition in Section \ref{main_res}).  Let $G(x)$ be the Gram function of this basis. Decompose an arbitrary function $f(x)$ in this basis
\[
f(x)=\sum_{k}\eta_{k}(f)\psi_{k}(x)
\]
where
\[
\eta_{k}(f)=\langle f,\psi_{k}\rangle_{G}=\int f(x)\psi_{k}(x)G(x)dx
\]
Now we can rewrite the expectation difference from the dual Wasserstein definition as
\[
\E f(X)-\E f(Y)=\langle f,\frac{\varphi_{X}}{G}\rangle_{G}-\langle f,\frac{\varphi_{Y}}{G}\rangle_{G}=\langle\eta(f),\theta(\varphi_{X})\rangle-\langle\eta(f),\theta(\varphi_{Y})\rangle
\]
where 
\[
\theta_{k}(\varphi)=\int\varphi(x)\psi_{k}(x)dx
\]
Define positive symmetric matrices 
\[
K_x =  \left(\begin{matrix}
    \nabla^T \psi_1(x) \\
    \ldots \\
     \nabla^T \psi_k(x) \\
     \ldots
\end{matrix} \right) 
\left(\begin{matrix}
    \nabla \psi_1(x) & \ldots & \nabla \psi_k(x) & \ldots 
\end{matrix} \right) = (\nabla^T \psi(x)) (\nabla \psi^T(x)) 
\]
and  
\[
K \circ G = \int K_x G(x) dx 
\]
Each $K_{x}$ is positive, since $\eta^{T}K_{x}\eta=\|\nabla f(x)\|^{2}$.
Condition $\forall x: \|\nabla f(x)\|\leq 1$ is equivalent
in Fourier basis to  
\begin{equation} \label{ellips_intersect}
\eta \in  \bigcap \mathcal{E}_x =\left\{ \eta : \forall x: \left( \sum_{k}\eta_{k}\nabla\psi_{k}(x) \right)^{2}=\eta^{T} K_x \eta \leq 1 \right\} 
\end{equation}
An important remark is that 
\begin{equation}\label{ellips_set}
\bigcap \mathcal{E}_x  \subset \left\{ \eta : \eta^{T} (K \circ G) \eta \leq 1 \right\} 
\end{equation}
We have got a useful intermediate result.
\begin{lemma}\label{wasfou} 
Let random vectors $X$ and $Y$ have
densities $\varphi_{X}$ and $\varphi_{Y}$ with Fourier coefficients $\theta_X$ and $\theta_Y$, then the Wasserstein distance  is the support function of the convex set $\bigcap \mathcal{E}_x$  (\ref{ellips_intersect}), i.e.  
\[
W_1( \varphi_X,  \varphi_Y)= \max_{\eta \in \bigcap \mathcal{E}_x }\langle\eta,\theta_X -\theta_Y \rangle
\]
Moreover, for the regularised case defined in Section \ref{main_res} it holds
\[
\widetilde{W_1} (\varphi_X, \varphi_Y)= \max_{\eta \in \bigcap \mathcal{E}_x }\langle\eta,\theta_X - \theta_Y\rangle - \varepsilon \eta^T  K \circ G \eta
\]
\end{lemma}
\noindent Remind that the barycenter's likelihood consists of independent components $l_i(\theta - \theta_i)$ with random vectors $\theta_i \in \R^{\infty}$ and parameter $\theta \in \R^{\infty}$,  such that   
\begin{align*}
l(\theta - \theta_i) & = \max_{\eta \in \bigcap \mathcal{E}_x }\langle\eta,\theta - \theta_i \rangle
- \varepsilon \eta^T  K \circ G \eta \\
& =  \max_{\eta}\langle\eta,\theta - \theta_i \rangle
- \varepsilon \eta^T  K \circ G \eta - \delta_{\bigcap \mathcal{E}_x}(\eta)
\end{align*}
Note that by definition  the dual function of $l$ is
\[
l^*(\eta)  = \delta_{\bigcap \mathcal{E}_x}(\eta) + \varepsilon \eta^T  K \circ G \eta 
\]
Consequently from Lemma \ref{ic_prop_1} follows that 
\[
l(\theta - \theta_i) = \delta^*_{\bigcap \mathcal{E}_x}(\theta - \theta_i) \oplus (\varepsilon \eta^T  K \circ G \eta)^*(\theta - \theta_i)  
\]
\begin{equation}
\label{l_conv}
=\left( \max_{\eta \in \bigcap \mathcal{E}_x }\langle\eta,\theta - \theta_i \rangle \right)
\oplus \frac{1}{\varepsilon} (\theta - \theta_i)^T  (K \circ G)^{-1} (\theta - \theta_i)
\end{equation}
Symbol $\oplus$ denotes infimal convolution (ref. definition in Section \ref{supf}).   Application of Theorem \ref{supdd},  taking into account property \ref{ellips_set}, provides the following bounds on the  derivatives of function $l$.
\begin{theorem} \label{wasd}  The gradient upper bound of the function $l$ is
  \[
\| D^{-1} \nabla l \| \leq \frac{1}{\lambda^{1/2}_{\min} \big( D  K\circ G D \big)}
\]
\end{theorem}

\begin{proof} Denote 
\[
\etas(\theta) = \argmax_{\eta \in \bigcap \mathcal{E}_x}  \eta^T \theta
\]
Use equation (\ref{l_conv}). By the consequence of Lemma \ref{infd} and Lemma \ref{supd} $\exists \theta_0$:
\[
\nabla l (\theta - \theta_i)  = \eta^*(\theta_0) 
\]
Since $\| (K\circ G)^{1/2} \etas \| \leq 1$, which follows from expression (\ref{ellips_set}),
\[
 \| D^{-1} \etas \|  = \| D^{-1} (K\circ G)^{-1/2} (K\circ G)^{1/2} \etas \|  \leq \| D^{-1} (K\circ G)^{-1/2} \|
\]
\end{proof}

\begin{theorem} \label{wasdd}
  Upper bounds of the second derivative of  function $l$ are
\[
\|  D^{-1} \partial \nabla^T l(\theta - \theta_i)  D^{-1} \| \leq 
 \frac{1}{\min_x \lambda_{\min}(D  K_x D) \| (K\circ G)^{-1/2} (\theta - \theta_i) \|} 
\]and
\[
\|  D^{-1} \partial \nabla^T l(\theta - \theta_i)  D^{-1} \| \leq 
 \frac{1}{ \varepsilon \lambda_{\min}(D  K \circ G D)  } 
\]
\end{theorem}

\begin{remark} Matrix $K_x$ may be singular which makes the first bound non-informative. The second bound comes from the regulariser  $\varepsilon \eta^T  K \circ G \eta $ and has big coefficient $(1 / \varepsilon)$.  It is a weak part of the current theory and requires an improvement or probably an example which shows that this bound it tight.
\end{remark}

\begin{proof} 
Consider support function with one ellipsoid. 
\[
s_x (\theta) =  \max_{\eta^{T} K_x \eta\leq 1}\langle\eta,\theta\rangle=\|K_x^{-1/2}\theta\|
\]
Denote $\eta^{*}(\theta) = \argmax \langle\eta,\theta\rangle $, and account that $\eta^{T} K_x \eta\leq 1$.
\[
\eta^{*}(\theta)=\frac{K_x^{-1}\theta}{\|K_x^{-1/2}\theta\|}
\]
\[
\frac{\partial\eta^{*}(\theta)}{\partial\theta}=\frac{K_x^{-1}\theta^{T}K_x^{-1}\theta-K_x^{-1}\theta\theta^{T}K_x^{-1}}{\left(\theta^{T}K_x^{-1}\theta\right)^{3/2}}
\]
For some vector $\| \gamma \| = 1$ by means of property $ \| a \|^2 \| b \|^2 \geq (a^Tb)^2 $ 
\[
\gamma^TK_x^{-1} \gamma \theta^{T}K_x^{-1}\theta - \gamma^T K_x^{-1}\theta\theta^{T}K_x^{-1} \gamma \leq \| K_x^{-1} \| \theta^{T}K_x^{-1}\theta
\]
\[
\left\| \frac{\partial\eta^{*}(\theta)}{\partial\theta} \right \| \leq \frac{\|K_x^{-1} \|}{\left(\theta^{T}K_x^{-1}\theta\right)^{1/2}}
\]
Apply Theorem \ref{supdd} that gives the first bound
\[
\| D^{-1} \partial \nabla^T l(\theta - \theta_i) D^{-1} \| \leq \max_x \left \|D^{-1}\frac{\partial\eta_x^{*}(\thetas_x)}{\partial\theta} D^{-1} \right \| \frac{s_x(\thetas_x)}{s(\theta - \theta_i)}
\]\[
\leq \frac{\max_x \|D^{-1} K_x^{-1} D^{-1}\|}{ \| (K\circ G)^{-1/2} (\theta - \theta_i) \| }
\]The second bound for this norm follows directly from Lemma \ref{infdd} and equation (\ref{l_conv}). 


\end{proof}


\begin{remark} Wasserstein distance also may be differentiated directly. Paper \cite{LimitWas} contains corresponded lemma about directional derivative. For directions $h_1, h_2$ it holds 
\[
 W_1'(\mu_X, \mu_Y)(h_X, h_Y)  =  \max_{(u,v) \in \Phi(\mu_X, \mu_Y)} - (\langle u,h_X \rangle + \langle v,h_X \rangle) 
\]
where
\[
\Phi = \{ (u,v): \langle u, \mu_X \rangle + \langle v, \mu_Y \rangle = W_1(\mu_X, \mu_Y), \, \forall (x,y):
u(x) + v(y) \leq \| x - y \|  \}
\]
\end{remark}

\noindent Now we prove three essential properties of the barycenters model that corresponds to Assumptions 1,2,3 from Section \ref{slt}.  \\
\noindent \textbf{Property 1:} Let $\theta \in \localr$ and $\theta_0$, $\theta'_0$ be some vectors on the line between $\theta$ and $\thetas$. 
\begin{align*}
\delta(\rr) &= 
\frac{1}{\rr} \normp{ - D^{-1}  \{ \nabla \E L(\theta) -  \nabla \E L(\thetas) \} - D (\theta - \thetas) } \\
& \leq \| D^{-1} \{ \nabla^2 \E L(\theta_0) - \nabla^2 \E L(\thetas) \} D^{-1} \| \\
& \leq  \| D^{-1} \{ \nabla^3 \E L(\theta'_0) D^{-1} \} D^{-1} \| \, \rr
\end{align*}
Let $q(\theta_i)$ be  distribution of each $\theta_i$  then
\[
 \nabla^3 \E_i L(\theta - \theta_i) = \sum_{i=1}^n \int \nabla^3 l(\theta - \theta_i) q(\theta_i) d\theta_i = - \sum_{i=1}^n \int \nabla^2 l(\theta - \theta_i) \times \nabla q(\theta_i) d\theta_i 
\]
\[
 \| D^{-1} \{ \nabla^3 \E L(\theta) D^{-1} \} D^{-1} \| \leq \int \| D^{-1}  \nabla^2  L(\theta - \theta_x)  D^{-1} \|  \| D^{-1} \nabla q(\theta_x) \| d\theta_x
\]
And from Theorem \ref{wasdd} one gets 
\[
\|  D^{-1} \nabla^2 l(\theta - \theta_i) D^{-1} \| \leq \frac{1}{ \varepsilon \lambda_{\min} \big(D K \circ G D \big) }
\]
Subsequently 
\begin{equation}
\label{bc_delta}
\delta(\rr) =  \frac{\rr n }{ \varepsilon \lambda_{\min} \big(D K \circ G D \big)} \int  \| D^{-1} \nabla q(\theta) \| d\theta  
\end{equation}
\noindent \textbf{Property 2:} From Theorem \ref{ass2cond} and Theorem \ref{wasdd} follows that one may set
\[
\textbf{v}^2  = \frac{  n }{ \varepsilon^2\lambda_{\min}^2 \big(D K \circ G D \big)} 
\]
and
\[
R = \frac{ 1 }{\varepsilon \lambda_{\min} \big(D K \circ G D \big)} 
\]
Then
\begin{align*}
\zz(t) 
& \leq    \textbf{v} (12 \sqrt{2} H_1 + \sqrt{2t}) + R (24 H_2 + 24 \sqrt{H_2 t} + t/3) \\
& \leq \frac{ \sqrt{n} (12 \sqrt{2} H_1 + \sqrt{2t}) + 24 H_2 + 24 \sqrt{H_2 t} + t/3  }{\varepsilon \lambda_{\min} \big(D K \circ G D \big)} 
\end{align*}

\noindent \textbf{Property 3:} Each model component $l(\theta - \theta_i)$ without regularisation is convex since 
\[
l(\lambda \theta_1 + (1 - \lambda) \theta_2 - \theta_i) =  l(\lambda (\theta_1 - \theta_i) + (1 - \lambda) (\theta_2 - \theta_i))
\]
\[
 = \max_{\eta \in \bigcap \mathcal{E}_x} \langle \eta, \lambda (\theta_1 - \theta_i) + (1 - \lambda) (\theta_2 - \theta_i)) \rangle 
\]
\[
\leq \max_{\eta \in \bigcap \mathcal{E}_x} \langle \eta, \lambda (\theta_1 - \theta_i) \rangle + \max_{\eta \in \bigcap \mathcal{E}_x} \langle \eta, (1 - \lambda) (\theta_1 - \theta_i) \rangle
\]
\[
 = \lambda l(\theta_1 - \theta_i) + (1 - \lambda) l(\theta_2 - \theta_i)  
\]
Note that the regularised $l$ is also convex as a composition of convex functions and the complete model $L$ is convex $(-\nabla^2 L > 0)$ as a positive aggregation of convex functions.  

Combination of these properties is used in the next proof  which gives the deviation bound of the parameter $\widehat{\theta}$. 

\noindent \textbf{Proof of Theorem \ref{wasF}}  Basing on Theorem \ref{devbound} and Properties 1,2,3 which correspond to Assumptions 1,2,3 from Section \ref{slt} we have
\[
\normp{ D (\opttheta - \theta^*) -  D^{-1} \nabla L (\thetas) } 
\leq \{\delta(\rr)  +  \zz(t)\} \rr = \rombt
\]
with probability $\hiprobt$. The values of $\delta(\rr)$ and $\zz(t)$ are estimated in Properties 1,2, where
\[
\delta(\rr) = \frac{ \rr n }{\varepsilon \lambda_{\min} \big(D K \circ G D \big)} \int  \| D^{-1} \nabla q(\theta) \| d\theta  
\]
\[
\zz(t) = \frac{ \sqrt{n} (12 \sqrt{2} H_1 + \sqrt{2t}) + 24 H_2 + 24 \sqrt{H_2 t} + t/3  }{\varepsilon \lambda_{\min} \big(D K \circ G D \big)} 
\]
Simplify the previous expressions. Account that according to the theorem conditions 
\[
\lambda_{\min} \big( D K\circ G D \big) \geq C_{D} n
\quad \text{and} \quad
  \int   \| D^{-1} \nabla q(\theta) \| d\theta  = \frac{C_q}{\sqrt{n}} 
\]
\[
\delta(\rr) = \frac{ \rr }{\varepsilon C_{D}} \frac{C_q}{\sqrt{n}} 
\]
\[
\zz(t) = \frac{ 12 \sqrt{2} H_1 + \sqrt{2t}   }{\varepsilon C_{D} \sqrt{n}} + 
O \left(\frac{H_2}{n}\right)
\]
Lemma \ref{wasd} provides bound $\forall \theta$
  \[
\| D^{-1} \nabla l(\theta) \| \leq \frac{1}{\lambda_{\min}^{1/2} \big( D  K\circ G D \big)}
\]
From this bound and Hoefding's inequality (Theorem 2.8 \cite{Concentration}) follows the bound for $\| D^{-1} \nabla L(\thetas) \|$. With probability $\hiprobt$ 
\[
\rr \leq 4 \|D^{-1} \nabla L(\thetas) \| \leq  \frac{ 8 \sqrt{n} (1 + \sqrt{2t}) }{\lambda_{\min}^{1/2} \big( D K\circ G D \big)}
\leq \frac{ 8  + 8 \sqrt{2t} }{ \sqrt{C_{D}} }
\]Finally taking into account definition (\ref{romb})
\[
\rombt \leq  \frac{ O(H_1 + t)   }{\varepsilon C^{3/2}_{D} \sqrt{n}} +   \frac{  O(t  C_q ) }{\varepsilon C^2_{D} \sqrt{n}}   + O \left(\frac{H_2}{n}\right)
\]

\noindent \textbf{Proof of Theorem \ref{was_gar}}  Bind Theorems \ref{wasF} and \ref{BESW}.  Form Theorem \ref{devbound} follows that the bound of Theorem \ref{wasF} also holds for projection of the parameter~$\theta$:
\[
\|  \breve{D} (\opttheta_p - \thetas_p) -  \breve{D}^{-1}\scorer L(\thetas) \| \leq \rombt 
\]
So  with probability $\hiprobt$
\[
W_1 ( \breve{D} (\opttheta_p - \thetas_p), \, Z) = \min_{\pi(\opttheta, Z)} \E \| \breve{D} (\opttheta_p - \thetas_p) -  Z \| 
\]\[
\leq W_1 (\breve{D}^{-1}\scorer L(\thetas), \, Z)   + \rombt
\]
Furthermore  from  Theorem \ref{BESW}  follows 
\[
W_1(\breve{D}^{-1}\scorer L(\thetas), Z) \leq \sqrt{2} \mu_3 \left ( \log(6 p \sqrt{ \tr \{ \Sigma \} } ) - \log (\mu_3) \right)
\]
where $\Sigma = \Var [\breve{D}^{-1}\scorer L(\thetas)]$ and setting $X_i =\breve{D}^{-1}\scorer l(\thetas - \theta_i) $ with independent copy $X_i'$
\begin{align*}
\mu_3 &=  \sum_{i=1}^n  \E X_i^T \Sigma^{-1} X_i  \| X_i - X'_i \| \\
& \leq 2  \max \| X_i \| \sum_{i=1}^n  \E X_i^T \Sigma^{-1} X_i    
\end{align*}
Using the next two properties
\[
\sum_{i=1}^n  \E X_i^T \Sigma X_i     =  \tr \left\{ \Sigma^{-1} \sum_{i=1}^n  \E   X_i X_i^T \right\} = p   
\]
\[
\max \| X_i \| = \| \breve{D}^{-1}\scorer l(\thetas - \theta_i)  \| \leq \| D^{-1}\nabla l(\thetas - \theta_i)  \| 
\leq \frac{1}{ \lambda^{1/2}_{\min} \big( D  K\circ G D \big) }
\]
we get 
\[
\mu_3 \leq \frac{2p}{\sqrt{C_{D} n}}
\]
and 
\[
W_1(\breve{D}^{-1}\scorer L(\thetas), Z) \leq O \left(  \frac{p \log(np) }{\sqrt{C_{D} n}} \right)
\]

\noindent Analogically one can derive a consequence from Theorems \ref{wasF} and \ref{BES}.
Let $C_A$ be the anti-concentration constant of the distribution $\P(\| Z \| > z)$, defined in Theorem \ref{BES}. Then  $\forall z \in \R$
\begin{EQA}
 &&| \P ( \| \breve{D} (\opttheta_p - \thetas_p) \| >  z )  -  \P ( \| Z \| > z ) |   \\
 && \leq | \P ( \| \breve{D}^{-1} \scorer  L(\thetas) \| >  z )  -  \P ( \| Z \| > z  ) | 
 + C_A \rombt
\end{EQA}
and 
\[
| \P ( \| \breve{D}^{-1} \scorer L(\thetas) \| > z  )  -  \P ( \| Z \| >  z  ) | \leq  C_A \mu_3 O( \log^2n )
\]
\noindent As for the anti-concentration constant it can be estimated as proposed in paper \cite{gotze2019} (Theorem 2.7), where was derived $C_A= O(1/\sqrt{p})$ for Euclidean norm. Grouping all together we get 
\[
| \P ( \| \breve{D} (\opttheta_p - \thetas_p) \| >  z )  -  \P ( \| Z \| > z ) |  
\leq  O \left( \frac{\sqrt{p} \log^2 n}{\sqrt{C_{D} n}} + \frac{ \rombt}{\sqrt{p}} \right)
\]



 \section{Support functions}

 \noindent Bounds for the first and second derivatives of the likelihood of barycenters model (\ref{bc_model}) involves additional theory from convex analysis.  
\definition[*]{
Legendre--Fenchel transform of a function $f:X \to \overline{\R}$ or the convex conjugate function calls 
\[
f^*(y) = \sup_{x \in X} ( \langle x, y \rangle - f(x) )
\]
}
\definition[s]{ \label{supf}
Support function for a convex body $E$ is
\[
s_E(\theta) = \sup_{\eta \in E}  \theta^T \eta 
\]
}
\noindent Note that for the indicator function $\delta_E(\eta)$ of a convex set $E$ the conjugate function is support function of $E$, i.e.
\[
\delta^*_E(\theta) = s_E(\theta)
\]
\definition[$\oplus$]{ \label{inf_conv} Let $f_1, f_2:E \to \overline{\R}$ be convex functions. The infimal convolution of them is 
\[
(f_1 \oplus f_2)(x) = \inf_{x_1 + x_2 = x} (f_1(x_1) + f_2(x_2))
\]  
}
\begin{lemma} [Proposition 13.21 \cite{polo2}] \label{ic_prop_1}  Let $f_1, f_2:E \to \overline{\R}$ are convex lower-semi-continuous functions. Then
\[
(f_1 \oplus f_2)^* = f_1^* + f_2^* 
\]
\[
(f_1 + f_2)^* = \overline{f_1^* \oplus f_2^*}
\]
\end{lemma}
\begin{lemma} The support function of intersection $E = E_1 \cap E_2$ is infimal convolution of support functions for $E_1$ and $E_2$
\[
s_E(\theta) = \inf_{\theta_1 + \theta_2 = \theta} (s_{E1}(\theta_1) + s_{E2}(\theta_2))
\]
\end{lemma}
\begin{proof} According to the previous Lemma 
\[
\delta_{E_1 \cap E_2}(\eta) = \delta_{E_1}(\eta) + \delta_{E_2}(\eta),
\]
\[
(\delta_{E_1} + \delta_{E_2})^* = \overline{ \delta^*_{E_1} \oplus \delta^*_{E_2} }
\]
With additional property
\[
\text{intdom} \, \delta_{E_1} \cap \text{dom} \, \delta_{E_2} = \text{int} E_1 \cap E_2 \neq  \emptyset
\]
one have
\[
(\delta_{E_1} + \delta_{E_2})^* = \delta^*_{E_1} \oplus \delta^*_{E_2} 
\]
\end{proof}

\begin{lemma}\label{supd}  Let a support function $s_E(\theta)$ be differentiable, then its gradient belongs to the border of corresponded convex set $E$ 
\[
\nabla s_E(\theta) =  \etas(\theta)  \in \partial E
\]
where
\[
\etas(\theta) = \argmax_{\eta \in E} \eta^T \theta
\]
\end{lemma}
\begin{proof}
It follows from the convexity of $E$ and linearity of the optimization functional. 
\[
\frac{\partial \etas(\theta) }{\partial \| \theta \| }  = 0 \Rightarrow \frac{\partial \etas(\theta) }{\partial  \theta  }^T \theta = 0 
\]  
\[
\nabla s_E(\theta) = \frac{\partial \etas(\theta) }{\partial  \theta  }^T \theta  + \etas(\theta) = \etas(\theta)
\]

\begin{figure}
\begin{center}
\includegraphics[scale=0.4]{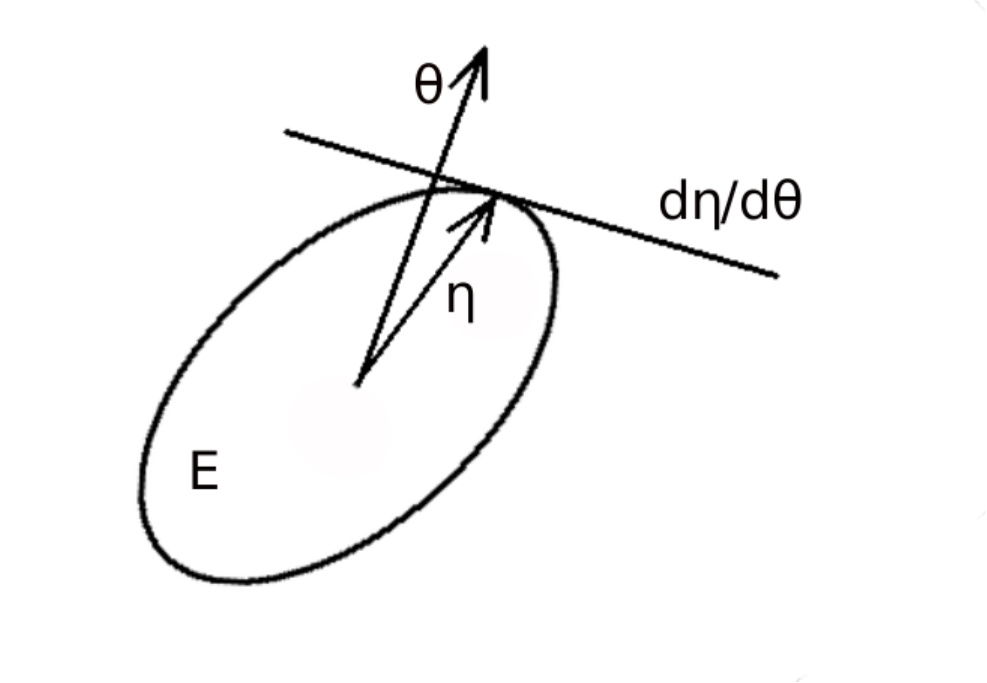}
\end{center}
\caption{Optimization related to support function.}
\end{figure}

\end{proof}

\begin{lemma}[Proposition 16.48 \cite{polo2}] \label{infd} 
   Let $f_1, f_2:E \to \overline{\R}$ be convex continuous functions. Then  the subdifferential of their infimal convolution can be computed by formula 
  \[
    \partial (f_1 \oplus f_2)(x) = \bigcup_{x = x_1 + x_2} \partial f(x_1) \cap \partial f(x_2)
  \]
\end{lemma}

\begin{corollary}
If in addition $f_1, f_2$ are differentiable, then their infimal convolution is differentiable and $\exists x_1, x_2: x = x_1 + x_2$ and 
\[
\nabla (f_1 \oplus f_2)(x) = \nabla f_1(x_1) = \nabla f_2(x_2)
\]
\end{corollary}

\begin{lemma}\label{infdd} Let $f_1,\ldots, f_m:E \to \overline{\R}$ be convex and two times differentiable functions. There is an upper bound for the second derivative of the infimal convolution $\forall t: \sum_{i=1}^m t_i = 1$
\[
\partial \nabla^T (f_1 \oplus \ldots \oplus f_m)(x) \preceq \sum_{i=1}^m t_i^2 \nabla^2 f(x_i)
\]
where $\sum_{i=1}^m x_i  = x$.
\end{lemma}

\begin{proof}
Use notation $f = f_1 \oplus \ldots \oplus f_m$. Let
\[
f(x) = \sum_i f_i(x_i)
\]
According to Lemma \ref{infd} if all the functions are differentiable then
\[
\nabla f(x) = \sum_i t_i \nabla f_i(x_i)
\]
From the definition $\oplus$ also follows that 
\[
f(x + z) \le \sum_i f_i(x_i + t_i z)
\]
Make Tailor expansion for the left and right parts and account equality of the first derivatives. 
\[
z^T \partial \nabla^T f(x + \theta z) z  
 \leq \sum_i  t_i^2 z^T \nabla^2 f_i(x_i + \theta_i z) z
\]
Since the direction $z$ was chosen arbitrarily, dividing both parts of the previous equation by $\|z\|^2 \to 0$, we come to inequality 
\[
  \partial \nabla^T f(x)   \preceq   \sum_i  t_i^2   \nabla^2 f_i(x_i)  
\]
\end{proof}

\begin{remark}
One can find another provement of the similar Theorem in book \cite{polo2} (Theorem 18.15).
\end{remark}

\begin{theorem}\label{supdd} Let $f_1,\ldots, f_m:E \to \overline{\R}$ be convex and two times differentiable functions. There is an upper bounds for infimal convolution $f = f_1 \oplus \ldots \oplus f_m$ derivatives $\forall \gamma$ $\exists x_1, \ldots, x_m$:
\[
 \gamma^T \partial \nabla^T f(x) \gamma  \leq \max_i \gamma^T \nabla^2 f_i(x_i) \gamma \frac{  f_i(x_i)  }{  f(x) }
\]
and
\[
\gamma^T \partial \nabla^T f^2(x)\gamma \leq 2 (\gamma^T \nabla f(x))^2  +  2  \max_i \gamma^T \nabla^2 f_i(x_i) \gamma  f_i(x_i) 
\]
\end{theorem}

\begin{proof}

Choosing appropriate $\{t_i\}$ in Lemma \ref{infdd} one get the required upper bounds. Set 
\[
t_i = \frac{f_i(x_i)}{\sum_{j=1}^m f_j (x_j) }
\]
and since  
\[
 \sum_{j=1}^m f_j(x_j) = f(x)
\]
\[
\sum_i  t_i^2  \gamma^T \nabla^2 f_i(x_i) \gamma  \leq \max_i t_i \gamma^T \nabla^2 f_i(x_i) \gamma = \max_i \gamma^T \nabla^2 f_i(x_i) \gamma f_i(x_i)   
\]
In order to prove the second formula apply this inequality in
\[
\partial \nabla^T f^2 = 2 \nabla f \nabla^T f + 2f \partial \nabla f
\]
\end{proof}

\begin{corollary} Let $s_1,\ldots, s_m: E^* \to \overline{\R}$ be support functions of the bounded convex smooth sets $E_1, \ldots, E_m$. There are upper bounds for the  derivatives of support function $s$ of intersection $E_1 \cap \ldots \cap E_m$, such that  $\forall i$
\[
\gamma^T \partial \nabla^T s(\theta) \gamma \leq \frac{ \max_i \gamma^T \partial \etas_i / \partial \theta_i  \gamma s_i (\theta_i)  }{  s ( \theta ) }
\]
\[
\gamma^T \partial \nabla^T s^2(\theta) \gamma \leq 2 (\gamma^T \etas_i)^2 +  2  \max_i \gamma^T \partial \etas_i / \partial \theta_i \gamma  s_i( \theta_i)
\]
\end{corollary}

\begin{proof}
It follows from Theorem \ref{supdd}  and Lemma \ref{supd}.
\end{proof}

 \section{Gaussian approximation}
\label{gar}
Multivariate analogues of Berry--Esseen Theorem have many modifications depending on space dimension of random vectors and   functions set used for measures comparison. V. Bentkus in his papers \cite{Bentkus}, \cite{Bentkus2003} has presented excellent results related to this topic. Namely for a sequence of i.i.d random vectors with identity covariance matrix $\{X_i\}_{i=1}^n$, $X_i \in \mathcal{P}(\R^p)$, any convex set $A$ and Gaussian vector $Z \in \ND(0, I)$ it holds 
\[
 \left | \P \left ( \frac{1}{\sqrt{n}} \sum_{i=1}^n X_i \in A \right) - \P  \left(Z \in A \right)  \right | \leq \frac{400 p^{1/4} \E \| X_1 \|^3}{\sqrt{n}} 
\]and
\[
W_1\left (\frac{1}{\sqrt{n}} \sum_{i=1}^n X_i , Z \right)  \leq O \left( \frac{ \E \| X_1 \|^3}{\sqrt{n}} \right) 
\]
We extend these two statements for independent random vectors with non-identity covariance $\Sigma$. Additionally we remove factor  $p^{1/4}$  replacing it with anti-concentration constant defined below.  
\definition[$H_k$]{
The multivariate Hermite polynomial is
\[
H_k(x) = (-1)^{|k|} e^{x^T \Sigma^{-1} x /2} \frac{\partial^{|k|}}{\partial^{k_1} \ldots \partial^{k_p}} e^{-x^T \Sigma^{-1} x/2}
\]where $x \in \R^p$ and $|k| = k_1 + \ldots + k_p$.
}

\begin{lemma} \cite{stein_lemma} \label{stein_sol} Consider a Gaussian vector $Z \sim \ND(0, \Sigma)$ and two functions $h \in C^1$ and $f_h$ such that 
\[
f_h(x) = - \int_0^1 \frac{1}{2t} \E \overline{h}(Z(x,t))dt
\]where for $t \in [0, 1]$
\[
\overline{h}(Z(x,t)) = h(\sqrt{t}x + \sqrt{1-t} Z) - \E h(Z)
\]
Then $f_h$ is a solution of Stein's equation
\[
\overline{h}(x) = (\tr\{\nabla^2 \Sigma \} - x^T \nabla) f_h(x)
\]
and
\[
\frac{\partial^{|k|}}{\partial^{k_1} \ldots \partial^{k_p}} f_h(x) = 
- \int_0^1 \frac{1}{2} \frac{t^{ \frac{|k|}{2}  - 1}}{(1 - t)^{\frac{|k|}{2}}} \E H_k(Z) \overline{h}(Z(x,t)) dt
\]
\end{lemma} 

\begin{proof} One may verify this statement through substituting the solution $f_h$ into Stein's equation. It has been done in Lemma 1 of paper \cite{stein_lemma}. 

\end{proof}

\noindent In the following discussion we will need difference between the second derivatives of function $f_h$.
\begin{corollary}
 Let $f_h$ be the solution of Stein's equation, then
\[
\nabla^2 f_h(x) - \nabla^2 f_h(y) = 
- \int_0^1  \frac{1}{2(1 - t)} \E H_2(Z) \{ h(Z(x,t)) - h(Z(y,t)) \}  dt
\]
where 
\[
H_2(Z) =  \Sigma^{-1/2} \{ (\Sigma^{-1/2}Z) (\Sigma^{-1/2}Z)^T  - I\} \Sigma^{-1/2}
\]
\end{corollary}

\noindent Use notation $\forall i$ $X_{-i}$ for sum of $\{ X_j \}_{j=1}^n$ without the $i$-th element and $X'_i$  for an \textit{independent copy} of $X_i$.  Use  the following notation for conditional expectation 
\[
\E_{-i} =  \E (\cdot | X_i, X'_i)
\]
\[
X_{-i} = \sum_{j = 1, j \neq i}^n X_j
\]

\begin{lemma} \label{general_gar}  Consider a Gaussian vector $Z \in \ND(0, \Sigma)$ and a sequence of independent zero-mean random vectors $X = \sum_{i=1}^n X_i$ in $\R^{p}$ with the same non-singular variance matrix 
\[
\E X X^T  =  \Sigma
\]Then for any function with bounded the first derivative $h \in C^1(\R^p)$
\[
\E h(X) - \E h(Z) \leq   A  \log \left(\frac{6B}{A} \right) 
\]
where 
\[
A = \sum_{i=1}^n \E X_i^T \Sigma^{-1} X_i \, A_i,
\quad
B = \sum_{i=1}^n \E X_i^T \Sigma^{-1} X_i \, B_i
\]
and $\forall \alpha > 0$ on interval $t \in [0, 1-\alpha]$
\[
 A_i \geq \frac{1}{\sqrt{t}} \sup_{\| \gamma \| = 1, \,\, \theta \in [0, 1]} \E_{-i} J_t(\gamma, \theta, X_i, X_i')
\]and on interval $t \in [1-\alpha, 1]$ for the same $\alpha$
\[
 B_i \geq \frac{1}{2 \sqrt{1-t}} \sup_{\| \gamma \| = 1, \, \, \theta \in [0, 1]} \E_{-i} J_t(\gamma, \theta, X_i, X_i')
\]and
\[
J_t(\gamma, \theta, X_i, X_i') = \{ (\gamma^T \Sigma^{-1/2} Z)^2  - 1 \} \{ h(Z(X_{-i} + \theta X_i,t)) - h(Z(X_{-i} + X'_i,t)) \}
\]where
\[
Z(x,t) = \sqrt{t}x + \sqrt{1-t} Z
\]
\end{lemma}

\begin{proof}

\noindent  From Lemma \ref{stein_sol} follows that for any function $h$ with the first bounded derivative 
\[
\E h(X) - \E h(Z) =   \E \tr\{\nabla^2 \Sigma \} f_h(X) - \E \sum_{i=1}^n X_i^T \nabla  f_h(X)
\]
Let $\theta$ be some value in $[0, 1]$. Decompose  $\nabla  f_h(X)$ by Taylor formula 
\[
\nabla  f_h(X) = \nabla  f_h(X_{-i}) + \nabla^2  f_h(X_{-i} + \theta X_i) X_i
\]Note that 
\[
\E \tr\{\nabla^2 \Sigma \} f_h(X) =  \E \sum_{i=1}^n X_i^T \nabla^2  f_h(X_{-i} + X_i') X_i
\]Substitute them into the first expression 
\[ 
 \E h(X) - \E h(Z) = \sum_{i=1}^n \E X_i^T \left \{  \nabla^2 f_h(X_{-i} + X'_i) -  \nabla^2 f_h(X_{-i} + \theta X_i  ) \right \} X_i  
\]
\[ 
 = \sum_{i=1}^n \E (\Sigma^{-1/2}X_i)^T \Sigma^{1/2} \left \{ \nabla^2 f_h(X_{-i} + X'_i) - \nabla^2 f_h(X_{-i} + \theta X_i  ) \right \} \Sigma^{1/2} \Sigma^{-1/2} X_i  
\]
From the consequence of Lemma \ref{stein_sol} use equality for the second derivative difference. For a unit vector $\| \gamma \| = 1$ and conditional expectation $\E_{-i} =  \E (\cdot | X_i, X'_i)$
\[
\gamma^T \E_{-i} \Sigma^{1/2} \left \{  \nabla^2 f_h(X_{-i}+ X'_i) - \nabla^2 f_h(X_{-i} + \theta X_i ) \right \} \Sigma^{1/2} \gamma
\]
\[
 =  \int_0^1  \frac{1}{2(1 - t)} \E_{-i} \{ ((\Sigma^{-1/2}Z)^T \gamma)^2  - 1 \} \{ h(Z(X_{-i} + \theta X_i,t)) - h(Z(X_{-i} + X'_i,t)) \}  dt
\]
\[
\leq \int_0^{1-\alpha} \frac{t^{1/2}}{2(1 - t)} A_i dt
+  \int_{1-\alpha}^1 \frac{1}{(1 - t)^{1/2}} B_i dt
\]
\[
\leq - \frac{A_i}{2} \log(\alpha) + 2B_i \sqrt{\alpha}	
\]Sum it with $X^T_i \Sigma^{-1} X_i$ and finally we obtain
\begin{align*}
 \E h(X) - \E h(Z) & \leq - \frac{A}{2} \log(\alpha)  + 2B \sqrt{\alpha}	\\
& \leq A \left(1 + \log \left(\frac{2B}{A} \right) \right) \\
& \leq A  \log \left(\frac{6B}{A} \right)
\end{align*}

\end{proof}

\begin{theorem}[Multivariate Berry--Esseen  with Wasserstein distance] \label{BESW}
 Consider a sequence of independent zero-mean random vectors $X = \sum_{i=1}^n X_i$ in $\R^{p}$ with a variance matrix 
\[
\E X X^T  =  \Sigma
\]
Then 1-Wasserstein distance between $X$ and Gaussian vector $Z \in \ND(0, \Sigma)$ has the following upper bound 
\[
W_1(X, Z) \leq \sqrt{2} \mu_3 \log\left(6 p \sqrt{ \tr ( \Sigma ) }  / \mu_3 \right) 
\]
where
\[
\mu_3 =  \sum_{i=1}^n  \E   X_i^T \Sigma^{-1} X_i  \| X_i - X'_i \|
\]
and each $X_i'$ is an independent copy of $X_i$.
\end{theorem}

\begin{remark}
In i.i.d case with $\Sigma = I_p$ 
\[
W_1(X, Z)  = O \left( \frac{p^{3/2} \log(n)}{\sqrt{n}} \right) 
\]
These is the same theorem with a different proof in paper \cite{Bentkus2003}.  
\end{remark}

\begin{proof} Basing on Lemma \ref{general_gar} we consider $h$ with  property  $\|  \nabla h( \cdot ) \| \leq 1$ and involve definition of $A_i$ and $B_i$. This property comes from the dual definition of $W_1$, ref. Section \ref{main_res}. We decompose $A_i$ extracting $\sqrt{t}(X_i - X'_i)$ and decompose $B_i$ extracting $\sqrt{1-t} Z$.  

\[
A_i = \| X_i - X'_i \| \, \E_{-i} | ( \gamma^T \Sigma^{-1/2}Z)^2  - 1 | 
\]
\[
B_i = \E_{-i} | ( \gamma^T \Sigma^{-1/2}Z)^2  - 1 | \, \| Z \| 
\]
Note that $( \gamma^T \Sigma^{-1/2}Z)^2$ has chi-square distribution $\chi_1$. And its variance equals $2$. So we obtain by means of Cauchy--Bunyakovsky inequality 
\[
 \E_{-i} | ( \gamma^T \Sigma^{-1/2}Z)^2  - 1 |  \leq \sqrt{2}
\]and 
\begin{align*}
\E_{-i} | ( \gamma^T \Sigma^{-1/2}Z)^2  - 1 | \, \| Z \| 
 \leq \sqrt{ \E | ( \gamma^T \Sigma^{-1/2}Z)^2  - 1 |^2} \sqrt{ \E \| Z \|^2 }  = \sqrt{2 \tr (\Sigma)}
\end{align*}Subsequently 
\[
A \leq \sqrt{2}   \sum_{i=1}^n  \E   X_i^T \Sigma^{-1} X_i  \| X_i - X'_i \| =  \sqrt{2} \mu_3
\]
and 
\[
B \leq  \sqrt{2 \tr (\Sigma)}   \sum_{i=1}^n  \E   X_i^T \Sigma^{-1} X_i  = p \sqrt{2 \tr (\Sigma)}
\]

\end{proof}

\noindent For the next result we will need the following technical lemma. 
\begin{lemma} \label{subg_exp}
Let a random variable $\varepsilon$ has a tail bound $\forall \xx \geq \xx_0$ 
\[
\P(\varepsilon > h(\xx)) \leq e^{-\xx}
\]
Then for a function $g: \R_+ \to \R_+$ with derivative $g': \R_+ \to \R_+$
\[
\E \Ind[\varepsilon > h(\xx_0)] g(\varepsilon) \leq g(h(\xx_0)) e^{-\xx_0} + \int_{\xx_0}^{\infty} e^{-\xx} g'(h(\xx)) h'(\xx) d\xx
\] 
In particular
\[
\E \Ind[\varepsilon > h(\xx_0)] \varepsilon \leq h(\xx_0) e^{-\xx_0} + \int_{\xx_0}^{\infty} e^{-\xx} h'(\xx) d\xx
\] 
\[
\E \Ind[\varepsilon > h(\xx_0)] \varepsilon^r \leq h(\xx_0)^r e^{-\xx_0} + r \int_{\xx_0}^{\infty} e^{-\xx} h(\xx)^{r-1} h'(\xx) d\xx
\] 
\end{lemma}

\begin{theorem}[Multivariate Berry--Esseen] \label{BES}
 Consider a sequence of independent zero-mean random vectors $X = \sum_{i=1}^n X_i$ in $\R^{p}$ with  a variance matrix 
\[
\E X X^T  =  \Sigma
\]
Let $\varphi: \R^{p} \to \R_{+}$ be some norm function (sub-additive and homogeneous)
and  with Gaussian vector $Z \in \ND(0, \Sigma)$ fulfills \textbf{anti-concentration} property, such that $\forall x \in \R_{+}$
\[
\P (\varphi(Z) > x) - \P (\varphi(Z) > x + \Delta) \leq C_A \Delta 
\] 
Then the measure difference between $X$ and Gaussian vector $Z$ has the following upper bound~$\forall x$
\begin{align*}
\left | \P (\varphi(X) > x) -  \P (\varphi(Z) > x) \right |
& \leq 22 C_A \mu_3 
\log \left( \frac{3 p}{C_A \mu_3} \right) \log \left( \frac{\sqrt{2 \E \varphi^2(Z)}  p}{ 10 C_A \mu^2_3} \right) \\
& \leq C_A \mu_3 \, O(\log^2 n)
\end{align*}
where
\[
\mu_3 =  \sum_{i=1}^n  \E  X_i^T \Sigma^{-1} X_i  2 \varphi( X_i )
\]

\end{theorem}

\begin{proof}

\noindent Make some preliminary computations. Define a smooth indicator function 
\[
g_{x,\Delta}(t) = \begin{cases}
               0, & t < x - \Delta \\
               (x - t)/\Delta, & t \in [x-\Delta, x] \\
               1, & t > x
            \end{cases} 
\]
Set $h = g_{x, \Delta} \circ \varphi$.
Denote the required bound by $\delta$:
\[
\sup_{x \in \R_+} \left | \P (\varphi(X) > x) -  \P (\varphi(Z) > x) \right | \leq \delta
\]
Note that from sub-additive property of the function $\varphi$ follows
\[
g_{x, \Delta} \big ( \varphi(X + dX) \big) \leq g_{x, \Delta}  \big( \varphi(X) + \varphi( d X) \big)
\]
and
\[
g'_{x, \Delta}(t) = \frac{1}{\Delta} \Ind [ x -\Delta \leq t \leq x ] 
\] 
By the anti-concentration property 
\[
\E g'_{x, \Delta}(\varphi(Z)) = \frac{1}{\Delta} \big( \P (\varphi(Z) > x - \Delta) - \P (\varphi(Z) > x) \big) \leq C_A  
\]And using definition of $\delta$
\begin{align}
\E g'_{x, \Delta}(\varphi(Z(X,t))) 
& \leq \frac{1}{\Delta} \big( \P (\varphi(Z) > x-\Delta) - \P (\varphi(Z) > x) \big) + \frac{2\delta}{\Delta} \\
\label{g_grad_b} & \leq C_A + \frac{2\delta}{\Delta}  
\end{align}
Now we will bound $\E_{-i} J_t(\gamma, \theta, X_i, X_i')$ required in Lemma \ref{general_gar}. Remind that by definition 
\[
J_t(\gamma, \theta, X_i, X_i') = \{ (\gamma^T \Sigma^{-1/2} Z)^2  - 1 \} \{ h(Z(X_{-i} + \theta X_i,t)) - h(Z(X_{-i} + X'_i,t)) \}
\]
where 
\[
Z(x, t) = \sqrt{t} x + \sqrt{1-t} Z
\]
For some $\theta' \in [0, 1]$ using sub-additivity of $\varphi$ and Taylor formula we get
\begin{align*}
&  h(Z(X_{-i} + \theta X_i,t))  - h(Z(X_{-i} + X'_i,t)) \\
& \leq g_{x, \Delta}  \big( \varphi(Z( X_{-i} + X'_i ,t)) + \varphi( \sqrt{t} (\theta X_i - X'_i )) \big) -  g_{x, \Delta}  \big(  \varphi(Z(X_{-i} + X'_i,t) \big)  \\
& \leq g'_{x, \Delta} \big( \varphi(Z(X_{-i} + X'_i,t)) + \theta' \varphi( \sqrt{t} (\theta X_i - X'_i )) \big)  \varphi( \sqrt{t} (\theta X_i - X'_i ) )  
\end{align*}Together with  (\ref{g_grad_b})
\[
 \E_{-i} h(Z(X_{-i} + \theta X_i,t))  - \E_{-i} h(Z(X_{-i} + X'_i,t)) \leq \sqrt{t} \left ( C_A + \frac{2\delta}{\Delta} \right ) \varphi(\theta X_i - X'_i )  
\]
Analogically one can obtain the same inequality for the opposite sign and subsequently we get inequality with module
\[
\left |  \E_{-i} h(Z(X_{-i} + X'_i,t)) - \E_{-i} h(Z(X' + \theta X_i,t)) \right | 
\leq \sqrt{t} \left ( C_A + \frac{2\delta}{\Delta} \right ) \varphi( \theta X_i - X'_i )
\]
Using  the previous expression and notation 
\[\varepsilon^2 = (\gamma^T \Sigma^{-1/2}Z)^2 \sim \ND^2(0, 1)\]
estimate the upper bound of $J_t$
\begin{align*}
\frac{1}{\sqrt{t}} \E_{-i}  J_t(\gamma, \theta, X_i, X_i')  
 \leq | \tau - 1 | \left ( C_A + \frac{2\delta}{\Delta} \right ) \varphi( \theta X_i - X'_i ) +  \E \Ind[\varepsilon^2 > \tau] |\varepsilon^2 - 1|
\end{align*}
\noindent For $\varepsilon \sim \ND(0, 1)$ we have
\[
\P(\varepsilon > \sqrt{2\xx}) \leq e^{-\xx}
\quad \text{and} \quad
\P(\varepsilon^2 > 2\xx + 2 \log2) \leq  e^{-\xx}
\]
and by means of  Lemma \ref{subg_exp} we get for all $\tau \geq 1$
\[
\E \Ind[\varepsilon^2 > \tau] |\varepsilon^2 - 1|  \leq 2 (\tau + 1) e^{-\tau / 2}
\]and 
\begin{align*}
A_i &= \sup_{\| \gamma \| = 1,\, \theta \in [0,1]} \frac{1}{\sqrt{t}} \E_{-i} J_t(\gamma, \theta, X_i, X_i') \\   
& \leq |\tau - 1| \left ( C_A + \frac{2 \delta}{\Delta} \right ) \varphi(\theta X_i - X'_i ) +   2 (\tau + 1) e^{-\tau / 2} 
\end{align*}
One should find optimal values for the arbitrary parameters $\Delta > 0$ and $\tau \geq 1$.  Setting $ \Delta = \delta / (2 C_A )$  and  $\tau = 2 \log ( 3 p / (C_A \mu_3) )$ we obtain that
\begin{align*}
A & = \sum_{i=1}^n \E X_i^T \Sigma^{-1} X_i \, A_i \\
 & \leq 5 |\tau - 1|  C_A  \mu_3 +   2 (\tau + 1) e^{-\tau / 2} p \\
& \leq 11 C_A \mu_3 \log \left( \frac{3 p}{C_A \mu_3} \right)
\end{align*}
We need also the other upper bound for $B_i$ when $t$ is close to $1$. 
\begin{align*}
& \E_{-i} \{ \varepsilon^2  - 1 \}   h(Z(X_{-i} + X'_i,t)) \\ 
&  = \E_{-i} \{\varepsilon^2  - 1 \}  \{ h( \sqrt{t} (X_{-i} + X'_i) +  \sqrt{1 - t} Z  ) - h( \sqrt{t} (X_{-i} + X'_i)) \}  \\
&  \leq \sqrt{1-t} \E_{-i} | \varepsilon^2  - 1 | \, | g'_{x,\Delta} | \, \varphi(Z) \\  
& \leq \frac{\sqrt{1-t}}{\Delta} \sqrt{2 \E \varphi^2(Z)}    
\end{align*}In the last expression we have applied Cauchy--Bunyakovsky inequality and the upper bound for $| g'_{x,\Delta} | \leq 1/ \Delta$.
Accounting condition $ \Delta = \delta / (2 C_A )$ one may derive that 
\[
B_i =  \frac{2 C_A }{\delta} \sqrt{2 \E \varphi^2(Z)}
\]and furthermore 
\[
B  = \sum_{i=1}^n \E X_i^T \Sigma^{-1} X_i \, B_i \leq  \frac{2 C_A }{\delta} \sqrt{2 \E \varphi^2(Z)} \, p
\]
In order to make step from $h$ expectation difference to the probabilities difference we will use the next inequality:
\[
\P (\varphi(X) > x) \leq \E h(X) = \E h (Z) +  \E h(X) - \E h (Z)   
\]
\[
\leq  \P (\varphi(Z) > x - \Delta) +   \E h(X) - \E h (Z)   
\]
\[
\leq  \P (\varphi(Z) > x) +   \E h(X) - \E h (Z) + C_A \Delta   
\]
Which gives
\[
\delta \leq | \E h(X) - \E h (Z)| + C_A \Delta   \leq  | \E h(X) - \E h (Z)| + \frac{\delta}{2}
\]\[
\delta \leq 2 | \E h(X) - \E h (Z)| 
\]
Basing on Lemma \ref{general_gar} we consider $h = g_{x, \Delta} \circ \varphi$ and we have already estimated the main values of $A$ and $B$. Assuming $\delta > A$ we obtain 
\begin{align*}
\delta & \leq  2A (\log(6B) - \log(A)  ) \\
& \leq  2 A \left (\log  (6B \delta) - \log(\delta) - \log(A) \right) \\
& \leq  2 A \left ( \log  (6B \delta)  - 2 \log(A)  \right) \\
& \leq 22 C_A \mu_3 
\log \left( \frac{3 p}{C_A \mu_3} \right) 
\log \left(  \frac{ p \sqrt{2 \E \varphi^2(Z)}  }{ 10 C_A \mu^2_3} \right)
\end{align*}

\end{proof}

\begin{remark}
In i.i.d case with $\Sigma = I_p$ and $\phi(x) = O ( \| x \| )$ 
\[
\left | \P (\varphi(X) > x) -  \P (\varphi(Z) > x) \right |  = O \left(  \frac{ p \log^2(n)} {\sqrt{n}} \right) 
\]
Note that  Lemma \ref{BES}  improves  the classical Multivariate Berry--Esseen Theorem \cite{Bentkus} for the case of norm functions $\phi(x) = O ( \| x \| )$. Namely, it answers the open question ``Whether one can remove or replace the factor $p^{1/4}$ by a better one (eventually by 1)''.    
\end{remark}

\begin{center}
***
\end{center}

The author thanks Prof. Roman Karasev, Prof. Vladimir Spokoiny and Prof. Dmitriy Dylov for discussion and contribution to this paper.

\bibliographystyle{alea3}
\bibliography{references}

\begin{thebibliography}{16}
\providecommand{\natexlab}[1]{#1}
\providecommand{\url}[1]{\texttt{#1}}
\providecommand{\urlprefix}{URL }
\expandafter\ifx\csname urlstyle\endcsname\relax
  \providecommand{\doi}[1]{doi:\discretionary{}{}{}#1}\else
  \providecommand{\doi}{doi:\discretionary{}{}{}\begingroup
  \urlstyle{rm}\Url}\fi
\providecommand{\eprint}[2][]{\url{#2}}

\bibitem[{Agueh and Carlier(2011)}]{BinWas}
Martial Agueh and Guillaume Carlier.
\newblock {Barycenters in the Wasserstein space}.
\newblock \emph{{SIAM Journal on Mathematical Analysis}} \textbf{43}~(2),
  904--924 (2011).
\newblock \urlprefix\url{https://hal.archives-ouvertes.fr/hal-00637399}.

\bibitem[{Bauschke and Combettes(2011)}]{polo2}
Heinz~H. Bauschke and Patrick~L. Combettes.
\newblock \emph{Convex Analysis and Monotone Operator Theory in Hilbert
  Spaces}.
\newblock Springer Publishing Company, Incorporated, 1st edition (2011).
\newblock ISBN 1441994661, 9781441994660.

\bibitem[{Bentkus(2003{\natexlab{a}})}]{Bentkus2003}
V.~Bentkus.
\newblock A new method for approximations in probability and operator theories.
\newblock \emph{Lithuanian Mathematical Journal} \textbf{43}~(4), 367--388
  (2003{\natexlab{a}}).
\newblock ISSN 1573-8825.
\newblock \doi{10.1023/B:LIMA.0000009685.65777.06}.
\newblock \urlprefix\url{https://doi.org/10.1023/B:LIMA.0000009685.65777.06}.

\bibitem[{Bentkus(2003{\natexlab{b}})}]{Bentkus}
V.~Bentkus.
\newblock On the dependence of the berry–esseen bound on dimension.
\newblock \emph{Journal of Statistical Planning and Inference}
  (2003{\natexlab{b}}).

\bibitem[{Bespalov et~al.(2020)Bespalov, Buzun and Dylov}]{brule}
Iaroslav Bespalov, Nazar Buzun and Dmitry~V. Dylov.
\newblock BrulÉ: Barycenter-regularized unsupervised landmark extraction
  (2020).
\newblock \eprint{2006.11643}.

\bibitem[{Bigot et~al.(2019)Bigot, Cazelles and Papadakis}]{PenWas}
J{\'e}r{\'e}mie Bigot, Elsa Cazelles and Nicolas Papadakis.
\newblock {Penalization of Barycenters in the Wasserstein Space}.
\newblock \emph{{SIAM Journal on Mathematical Analysis}} \textbf{51}~(3),
  2261--2285 (2019).
\newblock \urlprefix\url{https://hal.archives-ouvertes.fr/hal-01564007}.

\bibitem[{Bonneel et~al.(2016)Bonneel, Peyr{\'e} and Cuturi}]{bc_coords_proj}
Nicolas Bonneel, Gabriel Peyr{\'e} and Marco Cuturi.
\newblock Wasserstein barycentric coordinates: Histogram regression using
  optimal transport.
\newblock \emph{{ACM Transactions on Graphics}} \textbf{35}~(4), 71:1--71:10
  (2016).
\newblock \doi{10.1145/2897824.2925918}.
\newblock \urlprefix\url{https://hal.archives-ouvertes.fr/hal-01303148}.

\bibitem[{Boucheron~S.(2013)}]{Concentration}
Massart~P. Boucheron~S., Lugosi~G.
\newblock Concentration inequalities: A nonasymptotic theory of independence.
\newblock \emph{Oxford University Press}  (2013).

\bibitem[{Edwards(2011)}]{KRdual}
D.A. Edwards.
\newblock On the kantorovich–rubinstein theorem.
\newblock \emph{Expositiones Mathematicae} \textbf{29}~(4), 387 -- 398 (2011).
\newblock ISSN 0723-0869.
\newblock \doi{https://doi.org/10.1016/j.exmath.2011.06.005}.
\newblock
  \urlprefix\url{http://www.sciencedirect.com/science/article/pii/S0723086911000430}.

\bibitem[{G\"otze et~al.(2019)G\"otze, Naumov, Spokoiny and
  Ulyanov}]{gotze2019}
Friedrich G\"otze, Alexey Naumov, Vladimir Spokoiny and Vladimir Ulyanov.
\newblock Large ball probabilities, gaussian comparison and anti-concentration.
\newblock \emph{Bernoulli} \textbf{25}~(4A), 2538--2563 (2019).
\newblock \doi{10.3150/18-BEJ1062}.
\newblock \urlprefix\url{https://doi.org/10.3150/18-BEJ1062}.

\bibitem[{Kroshnin et~al.(2019)Kroshnin, Suvorikova and Spokoiny}]{g_was_clt}
Alexey Kroshnin, Alexandra Suvorikova and Vladimir Spokoiny.
\newblock Statistical inference for bures-wasserstein barycenters.
\newblock \emph{arXiv:1901.00226}  (2019).

\bibitem[{Meckes(2009)}]{stein_lemma}
Elizabeth~S. Meckes.
\newblock On stein's method for multivariate normal approximation.
\newblock \emph{High Dimensional Probability V: The Luminy Volume.}  (2009).

\bibitem[{Rippl et~al.(2016)Rippl, Munk and Sturm}]{LimitWasG}
Thomas Rippl, Axel Munk and Anja Sturm.
\newblock Limit laws of the empirical wasserstein distance: Gaussian
  distributions.
\newblock \emph{Journal of Multivariate Analysis} \textbf{151}, 90--109 (2016).
\newblock ISSN 0047-259X.
\newblock \doi{https://doi.org/10.1016/j.jmva.2016.06.005}.
\newblock
  \urlprefix\url{https://www.sciencedirect.com/science/article/pii/S0047259X16300446}.

\bibitem[{Sommerfeld and Munk(2018)}]{LimitWas}
Max Sommerfeld and Axel Munk.
\newblock Inference for empirical wasserstein distances on finite spaces.
\newblock \emph{Journal of the Royal Statistical Society: Series B (Statistical
  Methodology)} \textbf{80}~(1), 219--238 (2018).
\newblock \doi{https://doi.org/10.1111/rssb.12236}.

\bibitem[{Spokoiny(2017)}]{spok_pen}
Vladimir Spokoiny.
\newblock {Penalized maximum likelihood estimation and effective dimension}.
\newblock \emph{Annales de l'Institut Henri Poincaré, Probabilités et
  Statistiques} \textbf{53}~(1), 389 -- 429 (2017).
\newblock \doi{10.1214/15-AIHP720}.
\newblock \urlprefix\url{https://doi.org/10.1214/15-AIHP720}.

\bibitem[{Steinerberger(2021)}]{FouWas}
Stefan Steinerberger.
\newblock Wasserstein distance, fourier series and applications.
\newblock \emph{Monatshefte für Mathematik} \textbf{194} (2021).
\newblock \doi{10.1007/s00605-020-01497-2}.

\end{thebibliography}

\end{document}